\documentclass[a4paper]
{cedram-smai-jcm}

\usepackage{tikz}
\usepackage{subcaption}
\usepackage{mathtools}
\usepackage{multirow}

\newcommand{\C}{\mathbb{C}}

\newcommand{\R}{\mathbb{R}}

\newcommand{\N}{\mathbb{N}}

\newcommand{\Om}{\Omega}
\newcommand{\norm}[1]{\left\lVert#1\right\rVert}
\newcommand{\LO}{L^2(\Omega)}

\newcommand{\wt}{\widetilde}
\newcommand{\I}{\item[\normalfont(i)]}
\newcommand{\II}{\item[\normalfont(ii)]}

\newcommand{\conja}{\overline{a}}
\newcommand{\conjg}{\overline{\gamma}}

\newcommand{\qt}{\hat{q}}

\DeclarePairedDelimiterX{\Norm}[1]{\lVert}{\rVert}{#1}

\title[Overcoming the order barrier two in splitting methods]{Overcoming the order barrier two in splitting methods when applied to semilinear parabolic problems with non-periodic boundary conditions}

\author[R. Häberli]{\firstname{Ramona} \lastname{Häberli}}
\address{Université de Genève, Switzerland}
\email{ramona.haeberli@unige.ch}

  

\begin{document}

\begin{abstract}
In general, high order splitting methods suffer from an order reduction phenomena when applied to the time integration of partial differential equations with non-periodic boundary conditions. In the last decade, there were introduced several modifications to prevent the second order Strang splitting method from such a phenomena. In this article, inspired by these recent corrector techniques, we introduce a splitting method of order three for a class of semilinear parabolic problems that avoids order reduction in the context of non-periodic boundary conditions. We give a proof for the third order convergence of the method in a simplified linear setting and confirm the result by numerical experiments. Moreover, we show numerically that the high order convergence persists for an order four variant of a splitting method, and also for a nonlinear source term.
\end{abstract}

\maketitle


\section{Introduction} \label{intro}

Splitting schemes are a natural approach to integrate numerically in time differential equations~\cite{Hai1, Hun03, McL02}. They allow a seperate treatment of different terms of the original problem and can therefore be implemented more efficiently than traditional time integration methods~\cite{Des01}. The division in several sub-problems enables further the application of different discretization techniques. Moreover, splitting methods can preserve geometric properties such as positivity and invariant sets~\cite{Han12}, which is for instance important in chemical physics~\cite{Bad13} or computational biology~\cite{Ger02}. 

For $\Om \subset \R^d$ a bounded domain with smooth boundary and final time $T > 0$, we consider the following class of semilinear parabolic problems,
\begin{equation}
\begin{split}
\partial_tu(x,t)  &= Du(x,t) +f(x,u(x,t)), \quad x \in \Om,  t \in (0,T],  \\ 
Bu(x,t)& =b(x,t),  \quad x \in \partial\Om, t \in (0,T],\\
u(x,0)&=u_0(x), \quad x \in \Om,
\end{split}
\label{pde}
\end{equation}
where $D$ is a diffusion operator, $f$ a smooth source term and $u_0$ an initial data satisfying the conditions on the boundary of~\eqref{pde}. Here, $B$ denotes appropriate boundary conditions of Dirichlet, Neumann or Robin type and $b(x,t)$ is a smooth function on the boundary $\partial\Om$ of the domain $\Om$.

To solve problem~\eqref{pde} by operator splitting, we integrate in time independently the source equation
\begin{align}\label{f}
\partial_tu = f(u), 
\end{align}
and the diffusion equation
\begin{align}\label{D}
\partial_tu = Du \quad \text{in}\; \Om, \qquad Bu =b \quad \text{on}\; \partial\Om.
\end{align}
We denote by $\phi^f_{\tau}(u_0)$ and $\phi^{D}_{\tau}(u_0)$ the flows of~\eqref{f},~\eqref{D} respectively, with the corresponding time step $\tau>0$ and initial condition $u_0=u(0)$. A general splitting scheme with arbitrary coefficients $\alpha_1,\beta_1,\ldots,\alpha_m,\beta_m$, $m~\geq~1$, is defined by the following formula,
\begin{equation}
u_{n+1}= \phi^{f}_{\alpha_m\tau}\circ \phi^{D}_{\beta_m\tau}\circ \ldots \circ \phi^{f}_{\alpha_1\tau}\circ \phi^{D}_{\beta_1\tau}(u_n),
\label{splitgen}
\end{equation}
where $u_n$ denotes the approximation of $u(t_n)$ at time $t_n = n\tau$. As an example, we note the formal order two Strang splitting scheme
\begin{align} \label{strang}
\phi^{\text{StrangNaiv}}_{\tau} = \phi^f_{\frac{\tau}{2}}\circ\phi^{D}_{\tau}\circ\phi^f_{\frac{\tau}{2}}.
\end{align}
We distinguish between the \emph{formal order} of a method, i.e. the convergence order when applied to non-stiff ODEs and the observed \emph{order} when integrating in time parabolic problems. For general boundary conditions, splitting methods of formal order strictly higher than one suffer in general from an order reduction \cite{Hun03}. This phenomena occurs because the solution $u$ leaves the domain of the operator $D$ after applying the flow of the source term $\phi^f_{\tau}$ and therefore creates a loss of time regularity when applying the diffusion flow over small time steps.
In particular, for homogeneous Dirichlet boundary conditions $u=0$ on $\partial\Om$ and a solution-independent source function $f=f(x)$, which does not vanish at the boundary $\partial\Om$, the Strang splitting method~\eqref{strang} converges with a reduced order between one and two. This behaviour occurs even for exponential methods, which are integrated by means of Krylov algorithms~\cite{Nie12}, where no splitting is considered~\cite{Can18}.

In the last decade, several modifications were introduced to prevent method \eqref{strang} from such an order reduction. In~\cite{Alo19, Alo20}, they show order two convergence for a scheme discretized both in space and time, see also~\cite{Alo17} for Lawson methods and~\cite{Alo21} for arbitrarily high-order splitting methods for semilinear wave problems. Another correction technique to avoid order reduction is the projection of the intermediate solution to the domain of the operator $D$ before applying its flow $\phi^D_{\tau}$~\cite{Ber20, Ein15, Ein16}. While in~\cite{Ein15, Ein16}, one calculates the corrector function from the nonlinear source function $f$, one obtains it directly from the output of the flow $\phi^f_{\frac{\tau}{2}}$ in~\cite{Ber20} and requires therefore no additional evaluation of a possible costly nonlinearity $f$ or diffusion operator $D$. One time step of the modified version of the naive Strang splitting~\eqref{strang} of~\cite{Ber20} is defined by the composition
\begin{align} \label{strangcorr}
\phi^{\text{StrangCorr}}_{\tau} = \phi^f_{\frac{\tau}{2}}\circ\phi^{-q_n}_{\frac{\tau}{2}}\circ\phi^{D+q_n}_{\tau}\circ\phi^{-q_n}_{\frac{\tau}{2}}\circ\phi^f_{\frac{\tau}{2}},
\end{align}
where $\phi^{-q_n}_{\frac{\tau}{2}}$ is the exact flow of  $\partial_tu = -q_n$ and analogously, for $\phi^{D+q_n}_{\tau}$, we consider the modified diffusion problem $\partial_tu = Du + q_n$ with boundary conditions $Bu = b$. The flows for the correctors $q_n$ are projections in the sense of geometric numerical integration (see \cite[Chap.~IV.4]{Hai1}). The corrector itself is the solution of
\begin{equation} \label{q_strang}
Dq_n = 0 \quad \text{in}\; \Om, \qquad  q_n = \frac{2}{\tau}(\phi^f_{\tau/2}(u_n)-u_n) \quad \text{on}\; \partial\Om.
\end{equation}

The goal of the present article is to introduce a third order splitting scheme, which does not suffer from an order reduction when applied to parabolic problems, with analysis in a simplified linear setting. Thereby, inspired by the mentioned corrector techiques for the Strang splitting scheme, we modify a formal three order splitting method of the form~\eqref{splitgen}. However, for formal order strictly higher than two, there exist no such methods with only positive real coefficients $\alpha_i,$ and $\beta_i$ (\cite{Gol96}, see also \cite{Bla05}). This is an issue for non-time-reversible problems such as parabolic equations. Therefore, as proposed independently in~\cite{Cas09} and \cite{Han09b}, we consider splitting schemes with complex coefficients $\alpha_i,\beta_i \in \C$ with positive real part to achieve an order higher than two for the class of problems~\eqref{pde} with periodic boundary conditions. We also mention~\cite{Bla13}, where such splitting methods up to order 16 were constructed.

For the complex coefficient $a=\frac{1}{4}(1-\frac{1}{\sqrt{3}}i)$ and its conjugate $\conja=\frac{1}{4}(1+\frac{1}{\sqrt{3}}i)$, 
we consider the composition scheme
\begin{align}
\phi^{\text{C3Naiv}}_{\tau} = \phi^{\text{StrangNaiv}}_{2\conja\tau}\circ \phi^{\text{StrangNaiv}}_{2a\tau}
\label{c3comp}
\end{align}
of formal order three \cite{Suz90, Yos90}, see also \cite[Theorem~II.4.1]{Hai1}, where $\phi^{\text{StrangNaiv}}_{\tau}$ is one step of the Strang splitting method~\eqref{strang}. Note that one can rewrite~\eqref{c3comp} as a splitting method
\begin{align}
\phi^{\text{C3Naiv}}_{\tau} =\phi^f_{\conja\tau}\circ\phi^{D}_{2\conja\tau}\circ\phi^f_{c\tau}\circ\phi^{D}_{2a\tau}\circ\phi^f_{a\tau},
\label{c3}
\end{align}
with $c=\frac{1}{2}$.
Numerical experiments show that method~\eqref{c3} converges in general with a reduced order one when applied to the class of parabolic equations~\eqref{pde} with a general source term $f$. In particular, for homogeneous boundary conditions $b=0$ and a solution-independent source function, which vanishes at the boundary $\partial\Om$, it converges in general with order two and therefore, suffers from an order reduction, in contrast to the order two method \eqref{strang}, which avoids order reduction~\cite{Fao15}.

Since method~\eqref{c3} can be written as the composition~\eqref{c3comp}, an easy but naive choice for a modification would be a composition scheme of the form
\begin{align} \label{c3corrcomp}
\phi^{\text{C3Naiv2}}_{\tau} = \phi^{\text{StrangCorr}}_{2\conja\tau}\circ \phi^{\text{StrangCorr}}_{2a\tau}
\end{align}
of the corrected Strang methods, $\phi^{\text{StrangCorr}}_{\tau}$, given in \cite{Ber20} or \cite{Ein15, Ein16}. However, we observe numerically that method~\eqref{c3corrcomp} is only of order two when applied to the class of equations~\eqref{pde} and therefore does not improve from order two to three.

To avoid order reduction, we shall show that it is not enough to justify that $u$ stays in the domain of $D$.
In the present article, we introduce the corrected order three splitting method defined by
\begin{align} \label{c3corr}
\phi^{\text{C3New}}_{\tau} =\phi^f_{\conja\tau}\circ\phi^{-q^{(2)}_n}_{\conja\tau}\circ\phi^{D+q^{(2)}_n}_{2\conja\tau}\circ\phi^{-q^{(2)}_n}_{\conja\tau}\circ\phi^f_{c\tau}\circ\phi^{-q^{(1)}_n}_{a\tau}\circ\phi^{D+q^{(1)}_n}_{2a\tau}\circ\phi^{-q^{(1)}_n}_{a\tau}\circ\phi^f_{a\tau},
\end{align}
which involves two corrector functions $q^{(1)}_n$ and $q^{(2)}_n$. The corresponding flows $\phi^{-q^{(1)}_n}_{a\tau}$ and $\phi^{-q^{(2)}_n}_{\conja\tau}$ modify the respective intermediate numerical solutions $u$ and $Du$ such that they stay in the domain of $D$ before applying the flows of the diffusion operator. We prove that the numerical method $u_{n+1}=\phi^{\text{C3New}}_{\tau}(u_n)$, given by the splitting \eqref{c3corr}, fulfills in the case of a solution-independent source term the third order convergence estimate
\begin{equation} \label{main_ineq}
\norm{u_n-u(t_n)}_{\LO} \leq C\tau^3(1+|\log(\tau)|) \quad \text{for all}\; 0 \leq t_n \leq T,
\end{equation}
for all $\tau >0$ and $C$ a constant, which is independent on $\tau$ and $n$.

The mentioned order reduction phenomena we observe for splitting methods of formal order two and three also appear for higher order splitting schemes. For the complex coefficient $\gamma = \frac{1}{2}(1+\sqrt{3+2\sqrt{2}}i)$ and its conjugate $\conjg$, we consider the composition scheme
\begin{align} \label{c4}
\phi^{\text{C4Naiv2}}_{\tau} = \phi^{\text{C3New}}_{\conjg\tau}\circ\phi^{\text{C3New}}_{\gamma\tau},
\end{align}
which is of formal order four \cite[Theorem~II.4.1]{Hai1}, but converges in general with reduced order strictly smaller than four. Inspired by method~\eqref{c3corr}, we shall modify this naive scheme and show numerically fourth order convergence of the new variant when applied to the class of equation~\eqref{pde}.

The outline is the following. In Section~\ref{Af}, we introduce the analytical framework, which is used to give a convergence analysis of the splitting method~\eqref{c3corr}. In Section~\ref{method}, we define the two corrector functions $q^{(1)}_n$ and $q^{(2)}_n$ and explain in detail the construction of the corrected splitting \eqref{c3corr}. In Section~\ref{conv_anal}, we prove the third order estimate~\eqref{main_ineq} in a linear setting. In Section~\ref{num_exp}, we confirm the results from the convergence analysis by the illustration of two linear examples. Moreover, we show that the high order convergence seems to persist for an order four variant of a splitting scheme, and also for a nonlinear source term and solution-dependent boundary conditions.

\section{Framework on analytic semigroups} \label{Af}

In this section, we describe the analytical framework that we use in this article. Thereby, we follow the notation of \cite[Chap.~3]{Lun95} and \cite{Ber20}.
Let $\Om \subset \R^d$ be a bounded domain 
with $C^2$-boundary $\partial\Om$ or a polyhedral, convex domain and $T > 0$. We consider the class of semilinear parabolic equations \eqref{pde}.
Here, $D$ is a second order differential operator, defined by
\begin{equation} \label{D_formula}
D = \sum_{i,j = 1}^d a_{ij}(x)\partial_{ij} +\sum_{i= 1}^d b_{i}(x)\partial_{i} + c(x)I,
\end{equation}
where the matrix $(a_{ij}(x)) \in \R^{d\times d}$ is assumed to be symmetric and, for some $\lambda >0$ and for all $x \in \Om$, satisfies the uniform elliptic condition
$$
\sum_{i,j = 1}^d a_{ij}(x)\xi_i\xi_j \geq \lambda |\xi|^2, \quad \xi \in \R^d.
$$
For simplicity, we assume that $a_{ij} \in C^1(\overline{\Om})$ and $b_i, c$ uniformly continuous in $\Om$.
Furthermore, $B$ denotes boundary conditions $B = \gamma(x)I$ of type Dirichlet or alternatively,
$$
B = \sum_{i= 1}^d \beta_{i}(x)\partial_{i} + \gamma(x)I \qquad \text{with} \qquad \inf_{x \in \partial\Om}\Bigl|\sum_{i= 1}^d \beta_{i}(x)\nu_i(x)\Bigr|>0
$$ 
of type Neumann ($\gamma =0$), Robin respectively. We assume continuous differentiability for the coefficients $\beta_i, \gamma$ and we denote by $\nu(x)$ the exterior unit normal vector at $x\in \partial\Om$. 

For simplicity, we proceed our analysis in the Hilbert space $\LO$ and thus, unless otherwise specified, denote by $\norm{\,\cdot\,}$ the $\LO$ norm, as well as the associated Hilbert-Schmidt operator norm. 
We adopt the following notation for $v \in \LO$ from~\cite{Ber20}: For $k =0, 1, \ldots$ we write
$$
v = \mathcal{O}(\tau^k)\quad \text{if} \quad \norm{v}_{\LO} \leq C\tau^k,
$$
for some $C>0$ independent on $\tau >0$ assumed small enough. Throughout this article we denote by $C$ a positive constant independent on $\tau$ and $n$ but not necessarly the same at different passages.

We assume the source function to be twice continuously differentiable in a neighbourhood $U \subset~H^2(\Om)$ of the exact solution $u$ of~\eqref{pde}, $f \in C^2(U,H^2(\Om))$, that includes the case $f=f(x)$, which is considered in the convergence analyis, or a Nemytskii operator $f=f(x,u(x,t))$. Additionally, we consider heterogeneous, continuously differentiable boundary conditions $b \in C^1([0,T],H^2(\partial \Om))$. In this case, it is convenient for the analysis to reformulate problem~\eqref{pde}, such that we obtain an equation with zero boundary conditions. We follow the construction made in \cite{Ber20,Ein16}. Let $z \in C^1([0,T],H^2(\Om))$ with boundary conditions $Bz=b$ on $\partial\Om$ and initial condition $z(0)=z_0$. Define $\wt{u}=u-z$ which satisfies
\begin{equation} \label{pdez}
\partial_t\wt{u} = D\wt{u}+ f(\wt{u}+z) + Dz-\partial_tz \quad \text{in}\; \Om, \qquad  B\wt{u} =0\quad \text{on}\; \partial\Om, \qquad \wt{u}(0)=u_0-z_0.
\end{equation}
We emphasize that this lifting methodology, which allows to come back to a problem with homogenous boundary conditions, is only used for the analysis, not for implementations.
We define the linear operator $A$ with homogenous boundary conditions $Bu=0$ as
\begin{align}
Av= Dv \quad \forall v \in D(A)=\{u \in H^2(\Om): Bu = 0 \;\text{on}\ \partial\Om\}. \label{A}
\end{align}
By \cite[pages 91, 92]{Tho06}, $A$ is a densely defined linear operator on the Banach space $\LO$ such that the resolvent set $\rho(A)$ contains the closure of the set $\Sigma_{\theta} :=\{ z \in \C : z \neq 0, |\arg(z)| \leq \pi-\theta\}$,
for some $\theta \in (0,\frac{\pi}{2})$. Furthermore, for all $z \in  \overline{\Sigma}_{\theta}$, let 
\begin{align*}
\norm{(zI-A)^{-1}} \leq \frac{M}{|z|},
\end{align*}
for some constant $M > 0$. Under these assumptions, $-A$ is a sectorial operator and $A$ is the infinitesimal generator of the analytic semigroup $e^{tA}$, represented by
\begin{align} \label{etA}
e^{tA} = \frac{1}{2\pi i}\int_{\Gamma} e^{zt}R(z,A)\;\text{dz},
\end{align}
where $R(z,A)$ is the residual of $A$ and $\Gamma$ is the boundary of $\Sigma_{\theta}$ with imaginary part increasing along $\Gamma$. 
Since $A$ is a symmetric elliptic operator defined on $D(A)$, all its eigenvalues $\{\lambda_j\}_{j=1}^{\infty}$ are in $(-\infty, 0]$. Furthermore, the corresponding eigenfunctions $\{e_j\}_{j=1}^{\infty}$ build an orthonormal basis of $L^2(\Om)$ \cite[Theorem~6.5.1]{Eva97}. Due to the fact that $L^2(\Om)$ is a Hilbert space, for $v =\sum_{j \geq 1} v_je_j$, we have $\norm{v}^2 =~\sum_{j \geq 1} |v_j|^2$ by the equality of Parseval.

Furthermore, for all $\alpha \geq 0$, the operator $A$ fulfills the parabolic smoothing property (see e.g~\cite[Theorem~1.3.3]{Lor05})
\begin{align} \label{psp}
\norm{(-A)^{\alpha}e^{tA}} \leq \frac{C}{t^{\alpha}}, \quad t>0.
\end{align}
For $\alpha=0$, we have that $e^{tA}$ is bounded. Additionally, for $k \geq 1$, $A^ke^{tA}v$ is bounded for any $v \in D(A^k)$ \cite[Proposition~2.1.1]{Lun95}, where
\begin{align*}
D(A^k)=\{u \in H^{2k}(\Om): Bu =  BAu = \ldots = BA^{k-1}u = 0 \;\text{on}\ \partial\Om\}. 
\end{align*}
If we multiply $A$ by a constant $\lambda \in \C$ with $\Re(\lambda)>0$, \eqref{psp} remains true. Consider the more general version in form of~\cite[Proposition~5.1.1]{Lor05}. They use the property $\sup\{\Re(z): z \in \sigma(A)\} \leq 0$, which does not change for $\lambda A$. Therefore, also $e^{\lambda tA}$ is a bounded operator. Moreover, the definition~\eqref{etA} requires that for given initial condition $u_0$, problem~\eqref{D} can be integrated in time along the complex line $t=\lambda t'$ whenever $t'>0$ is small enough, see~\cite{Cas09, Han09b} in the context of periodic boundary conditions. Thus, $z$ as defined in \eqref{pdez} can be interpretated as the solution of the differential equation $\frac{d}{dt} e^{t A}y = Ae^{tA}y$. In the context of complex time steps, we assume $b$ and $z$ in \eqref{pdez} to be complex analytic in an open complex neighbourhood of $[0,T]$ with values in $H^2(\Om)$, and we assume the time step $\tau$ in \eqref{c3corr} to be small enough. Note that, if $b$ is time-independent, one can simply choose $z \in H^2(\Om)$ also time-independent, and we have no such restriction on the time stepsize. For $n=0,\ldots,N, N \in \N$, we set $z_n(s) = z(t_n+s)$. 

The exact solution of~\eqref{pde} can be expressed in terms of the analytic semigroup $e^{tA}$~\eqref{etA},
\begin{align*}
u(t_{n+1}) = z_n(\tau) + e^{\tau A}(u(t_n)-z_n(0)) + \int^{\tau}_0 e^{(\tau-s) A}(f(u(t_n+s))+Dz_n(s)-\partial_t z_n(s))\;\text{ds},
\end{align*}
and we assume $u \in C^2([0,T],H^2(\Om))$.

For the analysis, we consider linear problems with solution-independent source term $f=f(x)$. In the case of homogenous boundary conditions, the diffusion operator $D$ as in~\eqref{D_formula} corresponds to $A$, and we can reformulate the general problem~\eqref{pde} into
\begin{equation} \label{pdelin}
\begin{split}
\partial_tu(x,t)  &= Au(x,t) + f(x),\quad x \in \Om,  t \in (0,T],  \\ 
Bu(x,t)& =0,  \quad x \in \partial\Om, t \in (0,T],  \\
u(x,0)&=u_0(x), \quad x \in \Om.
\end{split}
\end{equation}
The exact solution $u$ of~\eqref{pdelin} is given by
\begin{align}
u(t_{n+1}) = e^{\tau A}u(t_n) + (e^{\tau A}-I)A^{-1}f.
\label{solex}
\end{align}
We consider an initial condition $u_0 \in D(A)$ and we assume that the source function satisfies the regularity
\begin{equation} \label{f_cond}
f\in H^4(\Om).
\end{equation} 
Precisely,~\cite[Theorem 4.6.3]{Eva97} states that if $g \in \LO$ and $v \in H^1_0(\Om)$ is the solution of the elliptic boundary-value problem
\begin{equation} \label{prob_ell}
Dv  = g\quad \text{in}\; \Om,   \qquad
Bv =0  \quad \text{on}\; \partial\Om,
\end{equation}
on a bounded domain $\Om$ with $C^2$-boundary, then $v \in H^2(\Om)$. Furthermore,  
\begin{equation} \label{prob_ell_ineq}
\norm{v}_{H^2(\Om)} \leq C \norm{g}_{\LO},
\end{equation}
where the constant $C$ depends only on $\Om$ and the coefficients of $D$. We refer to \cite[Theorem~2.1]{Nec67} for polyhedral, convex domains. We note that in the context of the corrected Strang splitting methods, one assumes $f \in H^2(\Om)$ for the same purpose \cite{Ber20, Ein15, Ein16}. 

\section{New order three splitting method} \label{method}

The goal of this section is to define the new splitting method~\eqref{c3}, which does not suffer from an order reduction when applied to the class of problems~\eqref{pde}. Thereby, before applying the flow $\phi^{D}_{\tau}$ of the diffusion operator $D$, we project the intermediate solutions, such that $u$ and also $Du$ are in the domain of $D$.
For $j=1, 2$ and $n=0,\ldots,N$, $N \in \N$, we introduce time steps 
$\tau_1 = a\tau$ and $\tau_2 = \conja\tau$
and intermediate solutions
$$
\omega^{(1)}_n = u_n \quad \text{and} \quad \omega^{(2)}_n =\phi^f_{c\tau}\circ\phi^{-q^{(1)}_n}_{a\tau}\circ\phi^{D_n+q^{(1)}_n}_{2a\tau}\circ\phi^{-q^{(1)}_n}_{a\tau}\circ\phi^f_{a\tau}(u_{n}).
$$
Then, we define functions $r^{(j)}_{n}$, for which we introduce the following condition on the boundary,
\begin{equation} \label{r}
B\biggl(r^{(j)}_n -\frac{D\phi^f_{\tau_j}(\omega^{(j)}_n)-D\omega^{(j)}_n}{\tau_j}\biggr) = 0 \quad \text{on}\; \partial\Om.
\end{equation}
In the interior of $\Om$, we choose $r^{(j)}_n$ such that 
\begin{equation} \label{rint}
Dr^{(j)}_n = 0. 
\end{equation}
Furthermore, for any $n =0,\ldots,N$, we define $q^{(j)}_n$ by the solution of the elliptic problem 
\begin{equation}\label{qint}
Dq^{(j)}_n = r^{(j)}_n \quad \text{in}\; \Om,
\end{equation} 
with boundary conditions
\begin{equation} \label{q}
B\biggl(q^{(j)}_n - \frac{\phi^f_{\tau_j}(\omega^{(j)}_n)-\omega^{(j)}_n}{\tau_j}\biggr)=0 \quad \text{on}\; \partial\Om.
\end{equation}
Actually, equation~\eqref{rint} with boundary conditions~\eqref{r}, as well as equation~\eqref{qint} with boundary conditions~\eqref{q} have unique solutions $r^{(j)}_n, q^{(j)}_n~\in~\LO$ \cite[Chap.~6]{Eva97}. 
Then, we define the corrected version of the splitting method~\eqref{c3} by
\begin{align*}
\phi^{\text{C3New}}_{\tau}(u_n) =\phi^f_{\conja\tau}\circ\phi^{-q^{(2)}_n}_{\conja\tau}\circ\phi^{D+q^{(2)}_n}_{2\conja\tau}\circ\phi^{-q^{(2)}_n}_{\conja\tau}\circ\phi^f_{c\tau}\circ\phi^{-q^{(1)}_n}_{a\tau}\circ\phi^{D+q^{(1)}_n}_{2a\tau}\circ\phi^{-q^{(1)}_n}_{a\tau}\circ\phi^f_{a\tau}(u_{n}).
\end{align*}
For $j=1, 2$, $\phi^{-q^{(j)}_n}_{\tau_j}(u^{(j)}_0)=u^{(j)}_0-\tau_j q^{(j)}_n$  is the exact flow of $\partial_tu^{(j)} = -q^{(j)}_n$, with the respective initial condition $u^{(j)}_0$ and time step $\tau_j$. We note that the corrector functions $q^{(1)}_n$ and $q^{(2)}_n$ do not coincide in general.
\begin{remark} ~\label{rem_qr} For a solution-independent source term $f$, the boundary conditions \eqref{r} and \eqref{q} turn into
\begin{align} \label{bc_qr}
B(r^{(j)}_n - Df)=0 \qquad \text{and} \qquad B(q^{(j)}_n-f)=0.
\end{align}
Additionally, the corrector functions are not anymore dependent on $n$. We hence denote $q=q^{(1)}_n = q^{(2)}_n$ and $r=r^{(1)}_n = r^{(2)}_n$. In the general case $f=f(u)$, the introduced conditions~\eqref{bc_qr} approximate~\eqref{r} and~\eqref{q}. Precisely, for $j=1, 2$, considering the second order approximation 
$$
\phi^f_{\tau_j}(\omega^{(j)}_n)=\omega^{(j)}_n+\tau_j f(\omega^{(j)}_n) + O(\tau_j^2),
$$
we observe that $Bq^{(j)}_n = \frac{1}{\tau_j}(\phi^f_{\tau_j}(\omega^{(j)}_n)-\omega^{(j)}_n) \approx Bf$ and similarly for $Br^{(j)}_n$.
\end{remark}
\begin{remark} \label{impl_qr}
In~\eqref{q_strang}, the choice $Dq_n = 0$ was made for simplicity. In the interior of the domain $\Om$, there are other possibilities to define the corrector $q_n$ at a reduced cost, as shown in~\cite{Ein18} in the context of corrected Strang splitting methods. Precisely, considering $\Om \subset \R^d$, solving~\eqref{q_strang} can be seen as a problem with $d-1$ degrees of freedom. Equivalently, one can remplace the operator $D$ in~\eqref{rint}-\eqref{qint} as long as assumptions~\eqref{prob_ell}-\eqref{prob_ell_ineq} are satisfied. Additionally, we approximate $u_n \approx b$ on $\partial\Om$, which avoids to calculate $D\omega^{(j)}_n$, see~\cite{Ber20} again in the context of the Strang splitting scheme. We refer to Remarks~\ref{rem_numf} and~\ref{rem_num} for more details.
\end{remark}
\begin{remark} \label{rem:c4new}
In Section~\ref{num_exp}, we show numerically that the introduced correction techniques for order three can be generalized to construct a splitting method of order four. Inspired by the definition of method~\eqref{c3corr}, we modifiy the splitting scheme~\eqref{c4} of formal order four, such that it converges numerically with full order four when applied to parabolic problems of type~\eqref{pde}. The new splitting method writes
\begin{equation} \label{c4corr}
\begin{split}
\phi^{\text{C4New}}_{\tau} = &\phi^f_{\conjg\conja\tau}\circ\phi^{-\qt^{(4)}_n}_{\conjg\conja\tau}\circ\phi^{D+\qt^{(4)}_n}_{2\conjg\conja\tau}\circ\phi^{-\qt^{(4)}_n}_{\conjg\conja\tau}\circ\phi^f_{\conjg c\tau}\circ\phi^{-\qt^{(3)}_n}_{\conjg a\tau}\circ\phi^{D+\qt^{(3)}_n}_{2\conjg a\tau}\circ\phi^{-\qt^{(3)}_n}_{\conjg a\tau}\circ\phi^f_{\conjg a\tau}\\
 &\phi^f_{\gamma\conja\tau}\circ\phi^{-\qt^{(2)}_n}_{\gamma\conja\tau}\circ\phi^{D+\qt^{(2)}_n}_{2\gamma\conja\tau}\circ\phi^{-\qt^{(2)}_n}_{\gamma\conja\tau}\circ\phi^f_{\gamma c\tau}\circ\phi^{-\qt^{(1)}_n}_{\gamma a\tau}\circ\phi^{D+\qt^{(1)}_n}_{2\gamma a\tau}\circ\phi^{-\qt^{(1)}_n}_{\gamma a\tau}\circ\phi^f_{\gamma a\tau},
\end{split}
\end{equation}
where we replace the corrector functions $q^{(i)}$ defined in~\eqref{r}-\eqref{q} by new correctors $\qt^{(i)}, i= 1, 2, 3, 4$. The corresponding flows of these correctors modify the respective intermediate numerical solutions such that not only $u$ and $Du$, but also $D^2u$ and $D^3u$ stay in the domain of the diffusion operator $D$ before applying its flow $\phi^D_{\tau}$. Note that for a solution-independent source term $f$, the correctors $\qt^{(i)}$ coincide for all $i= 1, 2, 3, 4$, and we have only one corrector, analogously to method~\eqref{c3corr}, see Remark~\ref{rem_qr}.
\end{remark}

\section{Order three error estimates for the splitting method C3New} \label{conv_anal}
In this section, we prove that the splitting method~\eqref{c3corr} is convergent of global order three without order reduction when applied to problem~\eqref{pde} with solution-independent source term $f=f(x)$, as in~\eqref{f_cond}.
\begin{theorem} \label{globalerr_thm}
Consider problem~\eqref{pde} under the assumptions of Section~\ref{Af} with source function $f~\in~H^4(\Om)$. Then, the splitting method~\eqref{c3corr} satisfies the error estimate~\eqref{main_ineq},
\begin{equation*}
\norm{u_n-u(t_n)} \leq C\tau^3(1+|\log(\tau)|) \quad \text{for}\; 0 \leq t_n \leq T,
\end{equation*}
for all $\tau >0$ and $C$ a constant, which is independent on $\tau$ and $n$.
\end{theorem}
 We introduce the following hypothesis on the corrector function $q$.\\
\textbf{Hypothesis  on $q$.} For $A$ given in \eqref{A}, we define the two hypothesis,
\begin{description}
\item[\normalfont(H1)] \quad $f-q \in D(A)$ and $A(f-q) = \mathcal{O}(1)$,
\item[\normalfont(H2)] \quad $A(f-q) \in D(A)$ and $A^2(f-q) = \mathcal{O}(1)$.
\end{description}
For the corrector functions $q$ and $r$, defined as in Section~\ref{method}, we show that $f-q$ fulfills the hypothesis (H1) and (H2).
\begin{lemma} \label{lemma_fq} If $q$ fulfills the properties \eqref{r}$-$\eqref{q}, then \emph{(H1)} and \emph{(H2)} hold true. 
\end{lemma}
\begin{proof} By Remark~\ref{rem_qr}, we have boundary conditions~\eqref{bc_qr} for the corrector functions $r$ and $q$. Therefore, $v_1=Df-r \in H^2(\Om)$, and thus also $v_2 = f-q \in H^2(\Om)$ are solutions of the elliptic problem~\eqref{prob_ell}, with right hand sides $g_1=D^2f$ and $g_2=v_1$. We deduce that $Df-r, f-q \in D(A)$.
Furthermore, 
$$
A(f-q)=D(f-q)=Df-Dq=Df-r,
$$ and thus, $A(f-q) \in D(A)$. Additionally, using the $H^2$-regularity estimate~\eqref{prob_ell_ineq} for elliptic problems, there holds 
$$
\norm{A(f-q)}
\leq C \norm{D^2f} \quad \text{and} \quad  \norm{A^2(f-q)}
= \norm{D^2f},
$$ 
which are the desired estimates to satisfy the hypothesis (H1) and (H2).
\end{proof}
We are now in position to analyze the local error defined as $\delta_{n+1} = \phi^{\text{C3New}}_{\tau}(u(t_n))-u(t_{n+1})$ of the new splitting method~\eqref{c3corr}.
\begin{lemma} \label{localerr_lemma} Under the assumptions of Theorem~\ref{globalerr_thm}, the local error $\delta_{n+1}$ of the splitting method~\eqref{c3corr} fulfills
$$
\delta_{n+1} = \tau S(\tau A)(f-q),
$$
with $S(z) = ae^{z} + \frac{1}{2}e^{2\conja z} +\conja -z^{-1}(e^{z}-1)$ for $z \in \C$.
\end{lemma}
\begin{proof}
Firstly, we consider homogeneous boundary conditions $b=0$, i.e. the local error $\delta_n^0$ of problem~\eqref{pdelin}. In this case, a straightforward calculation shows that the numerical flow $\phi^{\text{C3New}}_{\tau}$ of method~\eqref{c3corr} is given by
\begin{align} \label{solnum}
\phi^{\text{C3New}}_{\tau}(u(t_n))
&= e^{\tau A}u(t_n) + \tau (ae^{\tau A}+\frac{1}{2}e^{2\conja\tau A}+\conja I)(f-q) +(e^{\tau A}-I)A^{-1}q.
\end{align}
From the formulas of the exact solution \eqref{solex} and the numerical solution \eqref{solnum}, we deduce
\begin{align} \label{localerr}
\delta^0_{n+1}
&=\tau\bigl(ae^{\tau A} + \frac{1}{2}e^{2\conja\tau A} +\conja -(\tau A)^{-1}(e^{\tau A}-I)\bigr)(f-q)
= \tau S(\tau A)(f-q).
\end{align}
In the context of general boundary conditions as introduced in Section~\ref{Af}, we obtain for the local error
\begin{align} \label{localerrb}
\begin{split}
\delta_{n+1}
&=\delta^0_{n+1} + z_n(\tau)-e^{2\conja\tau A}z_n(2a\tau)+e^{2\conja\tau A}z_n(2a\tau)-e^{\tau A}z_n(0)-z_n(\tau)+e^{\tau A}z_n(0)\\
&+ e^{2\conja\tau A}\int^{2a\tau}_0 e^{(2a\tau-s) A}r^{D,z}_n(s)\;\text{ds} + \int^{\tau}_{2a\tau} e^{(\tau-s) A}r^{D,z}_n(s)\;\text{ds} - \int^{\tau}_0 e^{(\tau-s) A}r^{D,z}_n(s)\;\text{ds} 
=\delta^0_{n+1},
\end{split}
\end{align}
where we denote $r^{D,z}_n(s) = Dz_n(s)-\partial_t z_n(s)$.
We have the same local error as in the homogeneous case $b=0$, for which we have the desired representation~\eqref{localerr}. This concludes the proof.
\end{proof}
To obtain an estimate for the local error~\eqref{localerrb}, we need the following lemma.
\begin{lemma} \label{lemma_S} The mapping $S(z)$, defined in Lemma~\ref{localerr_lemma}, fulfills the following inequalities for all $z~\in~\R$.
\begin{description}
\I \quad For $z \geq -1$,
$
|S(z)| \leq |z|^3e^z.
$
\II \quad For $z \leq-1$,
$
|S(z)| \leq \sqrt{3/2}|z|^3. 
$
\end{description}
\end{lemma}
\begin{remark}
The constant $\sqrt{3/2}$ in Lemma \ref{lemma_S} (ii) is not optimal. Numerically, we notice that $$
\sup_{z \in (-\infty, -1]} |S(z)z^{-3}|\simeq 0.005.$$
\end{remark}
\begin{ProofOf}{Lemma~\ref{lemma_S}}
Consider the series expansion
\begin{align*}
S(z)
= ae^{z} + \frac{1}{2}e^{2\conja z} +\conja -z^{-1}(e^{z}-1) 
= \sum_{k\geq 1}\frac{1}{k!}\alpha_kz^k,
\end{align*}
with coefficients $\alpha_k = a + 2^{k-1}\conja^k -\frac{1}{k+1}$. We observe that $\alpha_1 = \alpha_2 =0$ and thus,
\begin{align} \label{S_repr}
S(z) = z^3 \sum_{k\geq 0}\frac{1}{(k+3)!}\alpha_{k+3}z^k = z^3\wt{S}(z). 
\end{align} 
\begin{itemize}
\I \quad For $k \geq 3$, we have $|\alpha_k| \leq 1$,
\begin{align} \label{alpha}
|\alpha_k| &=|a + 2^{k-1}\conja^k -\frac{1}{k+1}| \leq |a|(1+(2|a|)^{k-1})+\frac{1}{k+1} \leq |a|(1+(2|a|)^2)+\frac{1}{4} 
\leq 1.
\end{align}
Therefore, for $z \geq 0$, we deduce by~\eqref{S_repr},
\begin{equation} \label{S_zgrosser1}
|S(z)| \leq z^3\sum^{\infty}_{k=0}|\alpha_{k}|z^k \leq z^3\sum^{\infty}_{k=0}\frac{z^k}{k!} =z^3e^z. 
\end{equation}
For $z \in (-1,0)$, we obtain similarly,
\begin{align*}
S(z)e^{-z} = \conja e^{-z} + \frac{1}{2}e^{-2az} +a +z^{-1}(e^{-z}-1)
= \sum_{k\geq 1}\frac{(-1)^k}{k!}\bar{\alpha}_kz^k = z^3\sum_{k\geq 0}\frac{(-1)^{k+3}}{(k+3)!}\bar{\alpha}_{k+3}z^k.
\end{align*}
We deduce from \eqref{alpha},
\begin{equation} \label{S_zgrosser2}
|z^{-3}S(z)e^{-z}| \leq \sum_{k\geq 0}\frac{1}{(k+3)!}|\bar{\alpha}_{k+3}| |z|^k \leq  \sum_{k\geq 0}\frac{1}{(k+3)!} = e-\frac{5}{2} \leq 1. 
\end{equation}
\II \quad We note that $|S(z)|^2=\Re(S(z))^2+\Im(S(z))^2$. For $z \leq -1$, we obtain for the real part $\Re(S(z))$,
\begin{align} \label{real}
\Re(S(z))^2
&= \left(\frac{1}{4}(e^z+1) +\frac{1}{2}e^{\frac{z}{2}}\cos(\frac{z}{\sqrt{12}})-z^{-1}(e^{z}-1)\right)^2 \leq \frac{6}{5},
\end{align}
where we use the identity
\begin{equation} \label{real_imag}
e^{2\conja z}
=  e^{\frac{z}{2}}(\cos(\frac{z}{\sqrt{12}})+i\sin(\frac{z}{\sqrt{12}})).
\end{equation}
Furthermore, also by means of~\eqref{real_imag}, we deduce for the imaginary part $\Im(S(z))$,
\begin{align} \label{imag}
\Im(S(z))^2 
&=\left(\Im(a)(e^z-1)+\frac{1}{2}e^{\frac{z}{2}}\sin(\frac{z}{\sqrt{12}})\right)^2
\leq \frac{3}{10}.
\end{align}
Therefore, for $z \leq -1$, we obtain by~\eqref{real} and~\eqref{imag} the estimate 
\begin{equation} \label{S_zkleiner}
|S(z)|  \leq \sqrt{\frac{6}{5}+\frac{3}{10}} = \sqrt{\frac{3}{2}}\leq \sqrt{\frac{3}{2}}|z|^3.
\end{equation}
\end{itemize}
To conclude, we have the desired estimates for $z\geq -1$ in~\eqref{S_zgrosser1} and~\eqref{S_zgrosser2} and for $z \leq -1$ in~\eqref{S_zkleiner}.
\end{ProofOf}
Hypothesis (H1) and (H2) together with Lemma \ref{lemma_S} give us the required properties for the local error. 
\begin{lemma} \label{prop1_localerr} Consider the assumptions of Theorem~\ref{globalerr_thm} and assume \emph{(H1)} and \emph{(H2)} hold true. Then, for the splitting method \eqref{c3corr}, the local error $\delta_{n+1}$ satisfies the following estimates,
\begin{equation} \label{delta1}
\delta_{n+1} = \mathcal{O}(\tau^3),
\end{equation}
and
\begin{equation} \label{delta2}
 \delta_{n+1} = A\mathcal{O}(\tau^4).
\end{equation}
\end{lemma}
\begin{proof} 
Let $v \in D(A^2)$ and consider its expansion in the Hilbert basis, $v = \sum_{j\geq 1}v_je_j$. Denote by $\{\lambda_j\}_j$ the corresponding eigenvalues of $A$, then $\wt{S}(\tau A)v = \sum_{j\geq 1}\wt{S}(\tau\lambda_j)v_je_j$, where $\wt{S}(z)$ was introduced in~\eqref{S_repr}. 
Let $\kappa \geq 1$ depending on $\tau$ such that
$$\tau\lambda_1 \geq \ldots \geq \tau\lambda_{\kappa} \geq -1 > \tau\lambda_{\kappa+1} \geq \ldots.
$$ Thus, we obtain
\begin{align} \label{est1}
\Norm[\big]{(\tau A)^2\wt{S}(\tau A)v}^2 
=\sum_{j\geq1}^{\kappa}|(\tau\lambda_j)^2|^2|\wt{S}(\tau\lambda_j)|^2|v_j|^2 + \sum_{j\geq\kappa+1}|(\tau\lambda_j)^2|^2|S(\tau\lambda_j)|^2|v_j|^2.
\end{align}
By Lemma \ref{lemma_S} (i) and \eqref{psp} (with $\alpha=0$), we get the following bound,
\begin{align} \label{est2}
\sum_{j\geq1}^{\kappa}|(\tau\lambda_j)^2|^2|\wt{S}(\tau\lambda_j)|^2|v_j|^2 
&\leq \norm{(\tau A)^2e^{\tau A}v}^2  \leq \tau^4\Norm[\big]{e^{\tau A}}^2\norm{A^2v}^2 \leq C\tau^4.
\end{align}
Equivalently, for the second sum, we have by Lemma \ref{lemma_S} (ii),
\begin{align} \label{est3}
\sum_{j\geq \kappa+1}|(\tau\lambda_j)^2|^2|\wt{S}(\tau\lambda_j)|^2|v_j|^2 
\leq  \frac{3}{2}\sum_{j\geq \kappa+1}|(\tau\lambda_j)^2|^2|v_j|^2
=\frac{3}{2}\tau^4 \norm{A^2v}^2 \leq C\tau^4.
\end{align}
Therefore, Lemma \ref{lemma_fq} permits to conclude the proof of~\eqref{delta2} using $v=f-q$. For proving~\eqref{delta1}, we consider $S(\tau A)$ instead of $(\tau A)^2\wt{S}(\tau A)$ and proceed the same way as in~\eqref{est1}. Furthermore, instead of~\eqref{est2}, again by Lemma \ref{lemma_S} (i) and \eqref{psp} (with $\alpha=1$), we get
\begin{align*}
\sum_{j\geq1}^{\kappa}|(\tau\lambda_j)^3|^2|\wt{S}(\tau\lambda_j)|^2|v_j|^2 \leq \sum_{j\geq1}^{\kappa}|(\tau\lambda_j)^3|^2|e^{\tau\lambda_j}|^2|v_j|^2 \leq \tau^6 \norm{Ae^{\tau A}}^2\norm{A^2v}^2 \leq C\tau^4.
\end{align*}
For the second term in~\eqref{est1}, we use the same estimate as in~\eqref{est3}, which terminates the proof of~\eqref{delta1}.
\end{proof}
We may now prove the main result of this paper and show the global third order convergence for the splitting method \eqref{c3corr} by following the proof of \cite[ Proposition~4.10]{Ber20}.

\begin{ProofOf}{Theorem \ref{globalerr_thm}} The global error $e_{n} = u_n -u(t_n)$ satisfies
\begin{align*}
e_{n+1} = \phi^{\text{C3New}}_{\tau}(u_n) - u(t_{n+1})
=\phi^{\text{C3New}}_{\tau}(u_n)-\phi^{\text{C3New}}_{\tau}(u(t_n))+\delta_{n+1}.
\end{align*}
Applying \eqref{solnum}, we observe that $\phi^{\text{C3New}}_{\tau}(u_n)-\phi^{\text{C3New}}_{\tau}(u(t_n)) = e^{\tau A}e_n$. 
Thus, we have the following recursion formula for the global error,
\begin{equation*} 
e_n = e^{n\tau A}e_0 + \sum^{n-1}_{k=0}e^{(n-k-1)\tau A}\delta_{k+1}.
\end{equation*}
Since we choose the same initial condition for $u_n$ and $u(t_n)$, there holds $e_0 = 0$. Additionally, by Lemma~\ref{prop1_localerr},
\begin{align} \label{en_est1}
\sum^{n-1}_{k=0}e^{(n-k-1)\tau A}\delta_{k+1} = \sum^{n-2}_{k=0}e^{(n-k-1)\tau A}\delta_{k+1} + \delta_n = \sum^{n-2}_{k=0}e^{(n-k-1)\tau A}AO(\tau^4)+ O(\tau^3).
\end{align}
Furthermore, by the parabolic smooting property~\eqref{psp} of the operator $A$, we obtain
\begin{align} \label{en_est2}
\sum^{n-2}_{k=0}\Norm[\big]{e^{(n-k-1)\tau A}A} &\leq \frac{C}{\tau}\sum^{n-2}_{k=0}(n-k-1)^{-1} 
\leq  \frac{C}{\tau}\sum^{n-1}_{k=1}k^{-1} 
\leq \frac{C}{\tau}(1+|\log(\tau)|).
\end{align}
We obtain the desired estimate~\eqref{main_ineq} by means of~\eqref{en_est1} and~\eqref{en_est2}.
\end{ProofOf}
Theorem~\ref{globalerr_thm} states that the splitting method \eqref{c3corr} avoids order reduction when applied to equation~\eqref{pde} with source term $f=f(x)$. Since the naive method~\eqref{c3} suffers from an order reduction already when integrating in time problem~\eqref{pdelin} with $f$ vanishing on the boundary $\partial\Om$, we state the following result.
\begin{corollary} \label{globalerr_cor} Consider problem~\eqref{pdelin} under the assumptions of Section~\ref{Af} with source function $f~\in~H^4(\Om)$. Then, if $f \in D(A)$, the splitting method~\eqref{c3} satisfies the error estimate
\begin{equation} \label{main_ineq_two}
\norm{u_n-u(t_n)} \leq C\tau^2(1+|\log(\tau)|) \quad \text{for}\; 0 \leq t_n \leq T,
\end{equation}
for all $\tau >0$ and $C$ a constant, which is independent on $\tau$ and $n$. \\ 
If additionally $f \in D(A^2)$, then method~\eqref{c3} satisfies~\eqref{main_ineq}, that is
$$
\norm{u_n-u(t_n)} \leq C\tau^3(1+|\log(\tau)|) \quad \text{for}\; 0 \leq t_n \leq T.
$$
\end{corollary}
\begin{proof} Consider the local error $\delta^0_{n+1} = \phi^{\text{C3Naiv}}_{\tau}(u(t_n))-u(t_{n+1})$ of method~\eqref{c3} when applied to problem~\eqref{pdelin}. We obtain equivalently to~\eqref{localerr}, $\delta^0_{n+1}=\tau S(\tau A)f$. Furthermore, $f\in D(A)$ satisfies (H1) with $q=0$. Therefore, following the lines of the proof of Lemma~\ref{prop1_localerr} shows $\delta^0_{n+1} = A\mathcal{O}(\tau^3)$ and $\delta^0_{n+1} =\mathcal{O}(\tau^2)$. Finally, we get~\eqref{main_ineq_two} by following the proof of Theorem~\ref{globalerr_thm}.   
If additionally $f \in D(A^2)$, then (H1) and (H2) hold true for $q=0$, and we can proceed as for the proof of Theorem~\ref{globalerr_thm} with $f$ instead of $f-q$.
\end{proof}
In the next section, we confirm the results from Theorem~\ref{globalerr_thm} and Corollary~\ref{globalerr_cor} by numerical experiments. Furthermore, we give examples for more general problems which suggest that the global third order convergence of method~\eqref{c3corr} numerically also perstists for nonlinear source terms and solution-dependent boundary conditions.

\section{Numerical experiments} \label{num_exp}
In the following, we compare several splitting methods, where we use the notation \emph{StrangNaiv} for the Strang splitting~\eqref{strang} and \emph{StrangCorr} for the corrections made in~\cite{Ber20},~\cite{Ein15, Ein16} respectively, which coincide for a source term independent on the solution. Moreover, we denote by \emph{C3Naiv} the naive method~\eqref{c3} and by \emph{C3New} the new method~\eqref{c3corr}. For the methods of formal order four, we use \emph{C4Naiv} for the composition~\eqref{c4} and \emph{C4New} for the new method~\eqref{c4corr}. In the splitting schemes of formal order higher than two, complex time steps appear, which involve complex arithmetic. The extra computational costs are compensated by a higher accuracy, see Figure~\ref{fig:indf_1d_1}-\ref{fig:fisher}, \ref{fig:autbc}, \ref{fig:c4}.

\begin{table}[h]
\captionsetup{labelformat=simple,name=Table,labelfont=bf,textfont=it}
\caption{Convergence order $p$ of the naive splitting methods StrangNaiv and C3Naiv and the new method C3New applied to problem~\eqref{pdelin} for different source functions~$f$.\label{table_find}}%
\begin{tabular*}{14.9353cm}{|l|c|c|c|}
\hline
Splitting Method & $f \in D(A) \cap D(A^2)$ & $f \in D(A), f \notin D(A^2)$ & $f \notin D(A)$ \\
\hline
\multirow{2}{7em}{StrangNaiv~\eqref{strang}} & $p=2$, see~\cite{Ber20, Ein15, Ein16, Fao15} & $p=2$, see~\cite{Ber20, Ein15, Ein16, Fao15} & $1\leq p <2$\\	
 & no order reduction & no order reduction & order reduction \\
\hline
\multirow{2}{7em}{C3Naiv \eqref{c3}}& $p=3$, see Corollary~\ref{globalerr_cor} & $p=2$, see Corollary~\ref{globalerr_cor} & $1\leq p<2$\\
& no order reduction & order reduction & order reduction\\
\hline
\multirow{2}{7em}{C3New \eqref{c3corr}}& \multicolumn{3}{c|}{$p=3$, see Theorem~\ref{globalerr_thm}}\\
&\multicolumn{3}{c|}{no order reduction}\\
\hline
\end{tabular*}
\end{table}
Firstly, we apply the suggested correction \emph{C3New} of the splitting method \emph{C3Naiv} to several numerical examples in dimension $d=1,2$. On the one hand, we confirm the theory proved in Section \ref{conv_anal} by the performance of \emph{C3New} to equations of the form \eqref{pdelin}, i.e. problems with solution-independent source functions $f$. Thereby, we distinguish three different types of source terms (see Table~\ref{table_find}). On the other hand, we illustrate that \emph{C3New} does neither suffer from an order reduction when applied to nonlinear problems of type \eqref{pde}. In the splitting algorithm, the operator $D$ is approximated by finite differences and the diffusion problem \eqref{D} is solved exactly up to round off errors using the algorithm based on Krylov basis described in \cite{Nie12}. We use analytical formulas for the solution of the source equation \eqref{f} and for the flows of the correction functions $q_n$. 

\paragraph{A solution-independent source term in 1D (Figures \ref{fig:indf_1d_1} and \ref{fig:indf_1d_2})} In this first numerical experiment, we consider a parabolic equation on the interval $\Om = (0,1)$, with solution-independent source term $f=f(x)$ and Dirichlet boundary conditions,
\begin{equation} \label{pdeindep}
\partial_tu(x,t) = \partial_{xx}u(x,t)+ f(x), \quad  u(0,t)=u(1,t)=0, \quad u(x,0)=\sin(2\pi x).
\end{equation}
\begin{figure}[!tbp] 
\begin{minipage}[b]{0.49\textwidth}
\includegraphics[clip,trim=120 60 120 60,scale=0.36]{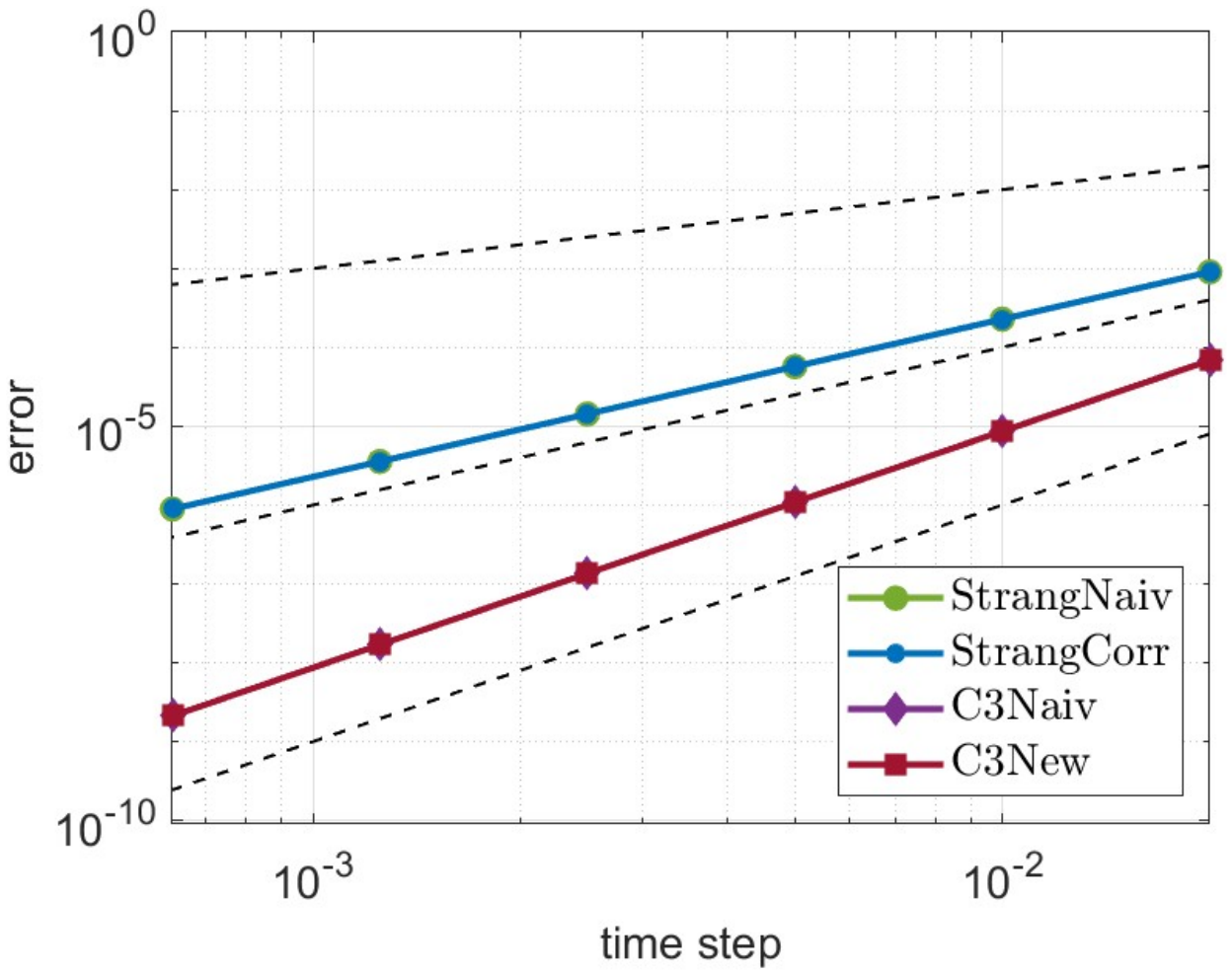}
\centering\small
$f(x)=\sin(2\pi x),$\; $f \in D(A) \cap D(A^2)$
\end{minipage}
\hspace{-10pt}
\begin{minipage}[b]{0.49\textwidth}
\includegraphics[clip,trim=120 60 120 60,scale=0.36]{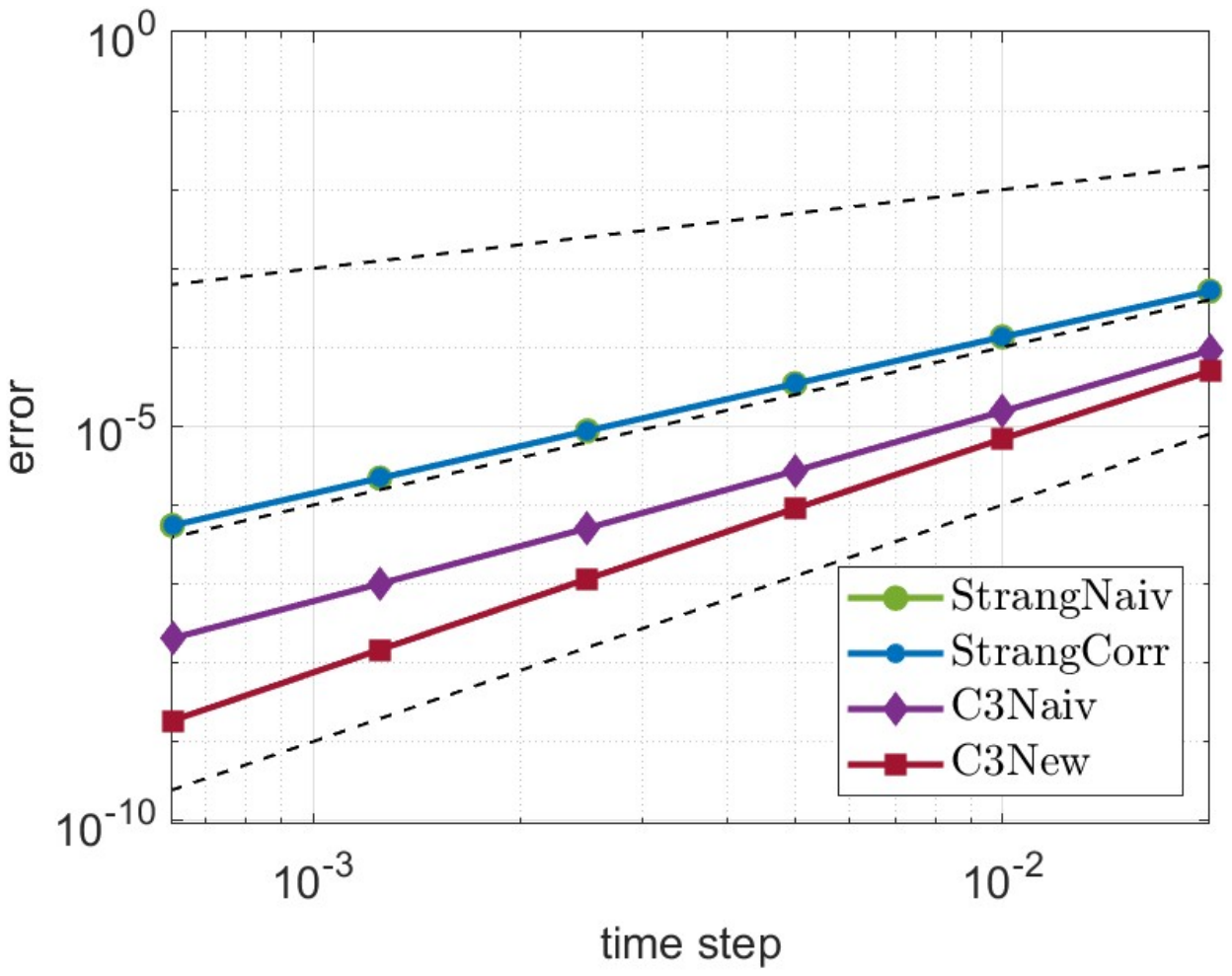}
\centering\small
$f(x)=x^2\sin(2\pi x),$\;$f \in D(A), f \notin D(A^2)$ 
\end{minipage}
\vspace*{1mm}
\caption{Convergence error of the corrected method \emph{C3New} applied to problem \eqref{pdeindep} for two different source functions inside the domain of the diffusion operator. Comparison to the naive third order method \emph{C3Naiv}, the Strang splitting \emph{StrangNaiv} and its  corrected version \emph{StrangCorr}. Reference slopes of order one, two and three are given in dotted lines. 
}
\label{fig:indf_1d_1}
\end{figure}

\begin{figure}[!tbp] 
\begin{minipage}[b]{0.49\textwidth}
\includegraphics[clip,trim=120 60 120 60,scale=0.36]{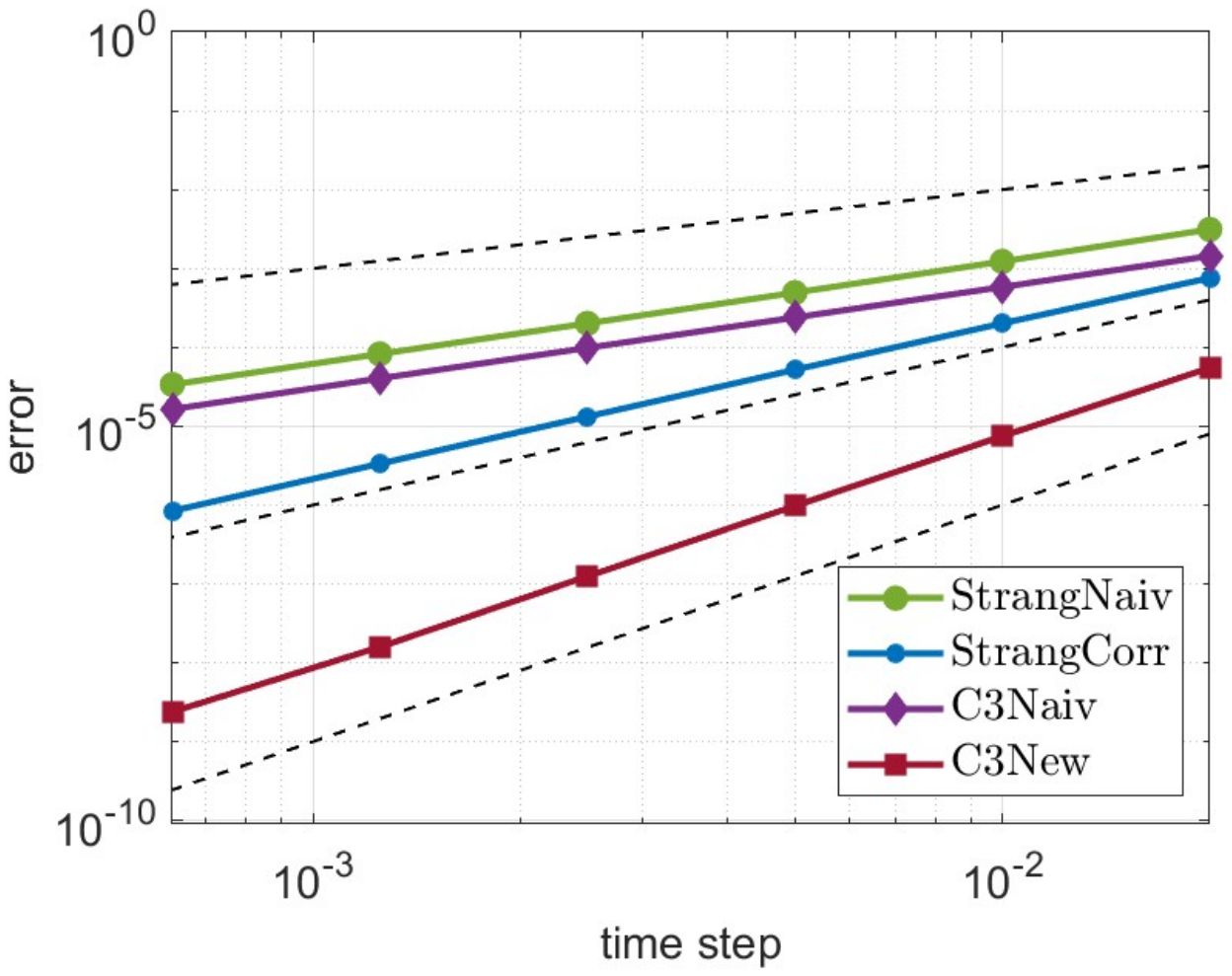}
\centering \small
$f(x)=\cos(2\pi x),$\;$f \notin D(A)$
\end{minipage} 
\hspace{-10pt}
\begin{minipage}[b]{0.49\textwidth}
\includegraphics[clip,trim=120 60 120 60,scale=0.36]{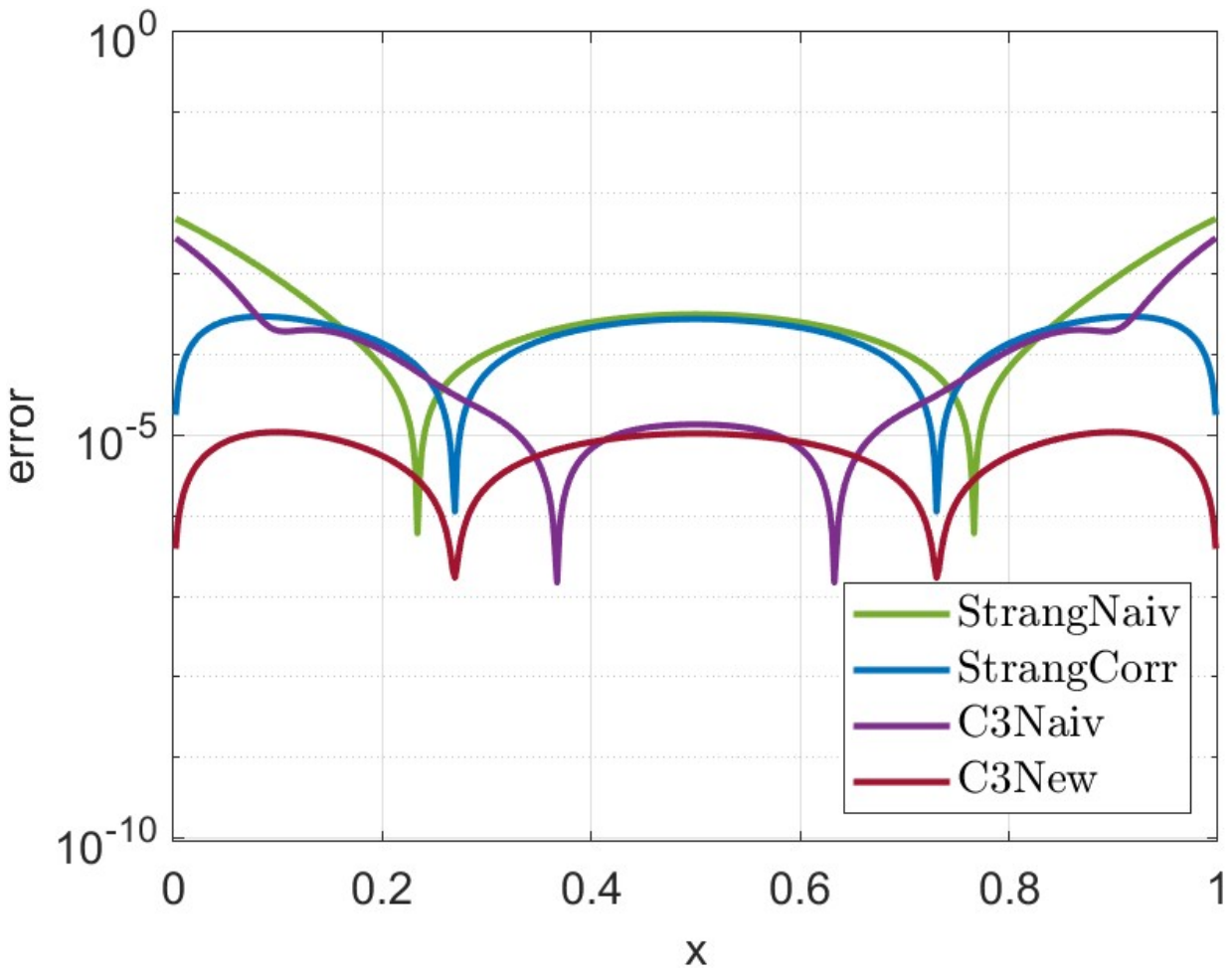}
\centering \small
\hspace{2cm}
\end{minipage}
\vspace*{1mm}
\caption{Convergence error of the corrected method \emph{C3New} applied to problem \eqref{pdeindep} for a source function $f(x)=\cos(2\pi x)$ outside the domain of the diffusion operator. Comparison to the naive third order method \emph{C3Naiv}, the Strang splitting \emph{StrangNaiv} and its corrected version \emph{StangCorr}. Left: Error in terms of the time step $\tau$. Reference slopes of order one, two and three are given in dotted lines. Right: The pointwise error at final time $T=0.1$ for time step $\tau = 10^{-2}$.
}
\label{fig:indf_1d_2}
\end{figure}

Problem~\eqref{pdeindep} was chosen to illustrate the results given in Section~\ref{conv_anal}. It can be solved exactly in time without any splitting. For the performance, we choose three different source functions $f$. The observed convergence orders are summerized in Table~\ref{table_find}. We mention that the method \emph{StrangNaiv} only suffers from an order reduction if we choose a general source function which is not in the domain of the diffusion operator, while \emph{C3Naiv} converges with reduced order if $f \in D(A)$ but $f \notin D(A^2)$, see Figure~\ref{fig:indf_1d_1} and~\ref{fig:indf_1d_2} (left). Moreover, the new method \emph{C3New} converges with order three in any case and has a better error constant compared to the other splitting methods we processed. Furthermore, from Figure~\ref{fig:indf_1d_2} (right) we deduce that the error for the naive methods \emph{StrangNaiv} and \emph{C3Naiv} is maximal at the boundary of the domain, which was already observed in \cite{Han09b} in the context of the Strang splitting method \emph{StrangNaiv}. In contrast, the corrected method \emph{C3New} is $10^4$ times more accurate on the boundary than its naive version. Since the pointwise error is presented in a logaritmic scale, we observe some "bumps"in the points where the error changes its sign.

For this experiment, we set the final time $T=0.1$, and time steps $\tau=0.02\cdot 2^{-k}$ for $k=0,1,\ldots,5$. For the reference solution, we use the naive method \emph{StrangNaiv} with a very small time step $\tau_{ref} = 10^{-6}$. The domain $\Om = (0,1)$ is discretized in a uniform mesh with mesh size $\Delta x = 2\cdot10^{-3}$.

\begin{remark} \label{rem_numf} In the previous experiment~\eqref{pdeindep}, the source $f=f(x)$ and its derivatives are available with an explicit analytical formula, and therefore no numerical differentiation is required. Precisely, as stated in~\eqref{bc_qr}, we set $q=f$  and $r=f''$ on the boundary $\partial\Om$. Furthermore, to discretize the domain $\Om$, we consider the uniform mesh $\wt{\Om} = (x_i)_{i=1}^N$, where $N+1=1/\Delta x$. We use finite differences to approximate the Laplacian $\partial_{xx}$ in $\wt{\Om}$. Thus, for $i=1, \ldots, N$, the correctors $r$ and $q$ are obtained by the formulas
$$
\partial_{xx}r(x_i)+\delta_{i1}f''(0)+\delta_{iN}f''(1)=0 \quad \text{and}\quad \partial_{xx}q(x_i)+\delta_{i1}f(0)+\delta_{iN}f(1)=r(x_i),
$$
where $\delta_{ij}$ denotes the Kronecker delta. The derivative $f''$ can be evaluated analytically or numerically. Analogous calculations are made if we consider a similar problem in a two-dimensional domain.
\end{remark}

\paragraph{A solution-independent source term in 2D (Figure \ref{fig:indf_2d})} We consider the following parabolic problem with Dirichlet boundary conditions and solution-independent source term $f(x,y)$,
\begin{equation}
\begin{split}
\partial_tu(x,y,t) &= \Delta u(x,y,t)+ e^xy^7+1,\\ 
u(0,y,t)&=y, \quad u(1,y,t)=1+y \quad u(x,0,t)=x, \quad u(x,1,t)=1+x,\\
u_0(x,y)&=x+y.
\label{pdeindep2d}
\end{split}
\end{equation}
\begin{figure}[!tbp] 
\begin{minipage}[b]{0.49\textwidth}
\includegraphics[clip,trim=100 60 120 60,scale=0.36]{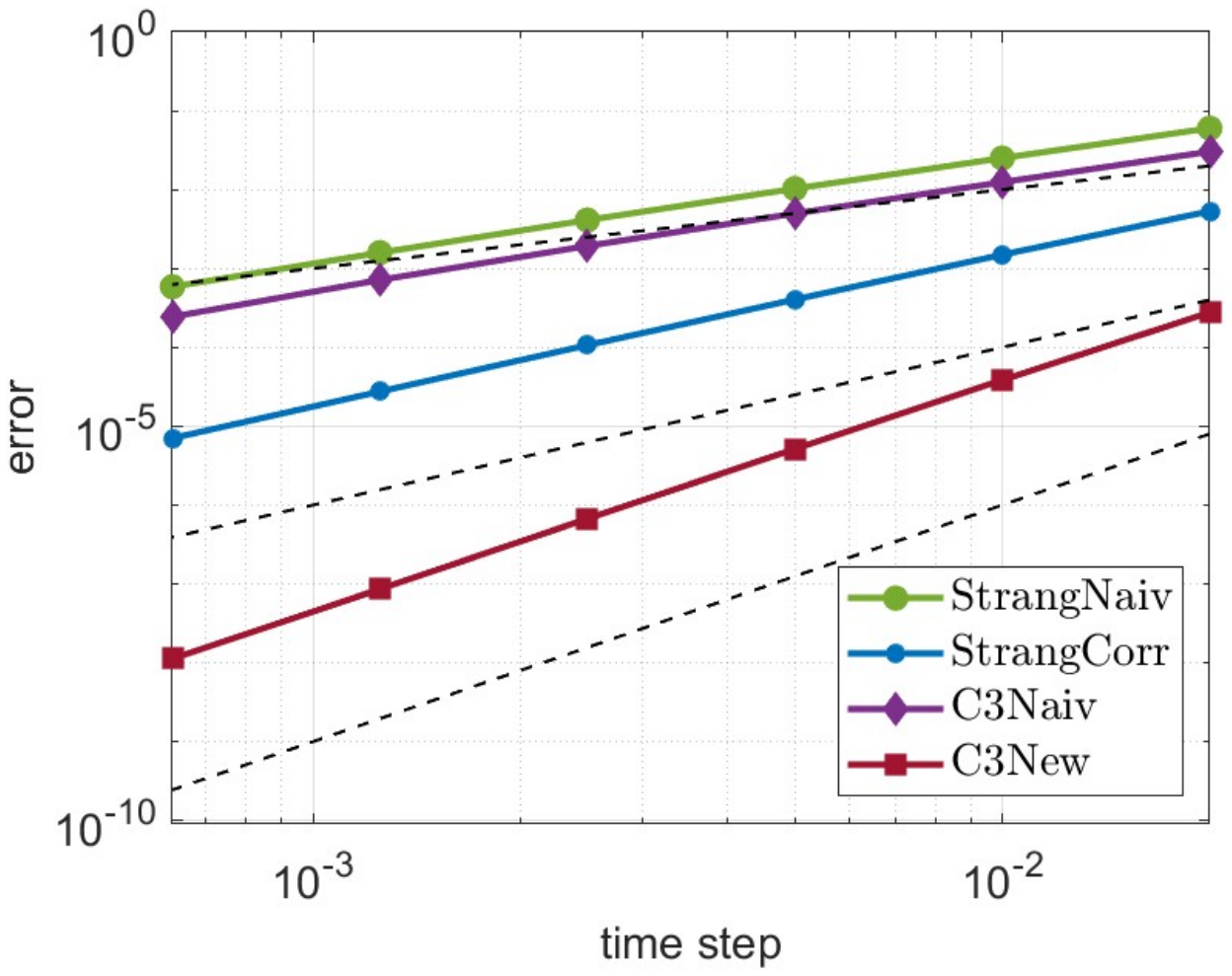}
\end{minipage}
\hspace{-10pt}
\begin{minipage}[b]{0.49\textwidth}
\includegraphics[clip,trim=100 60 120 60,scale=0.36]{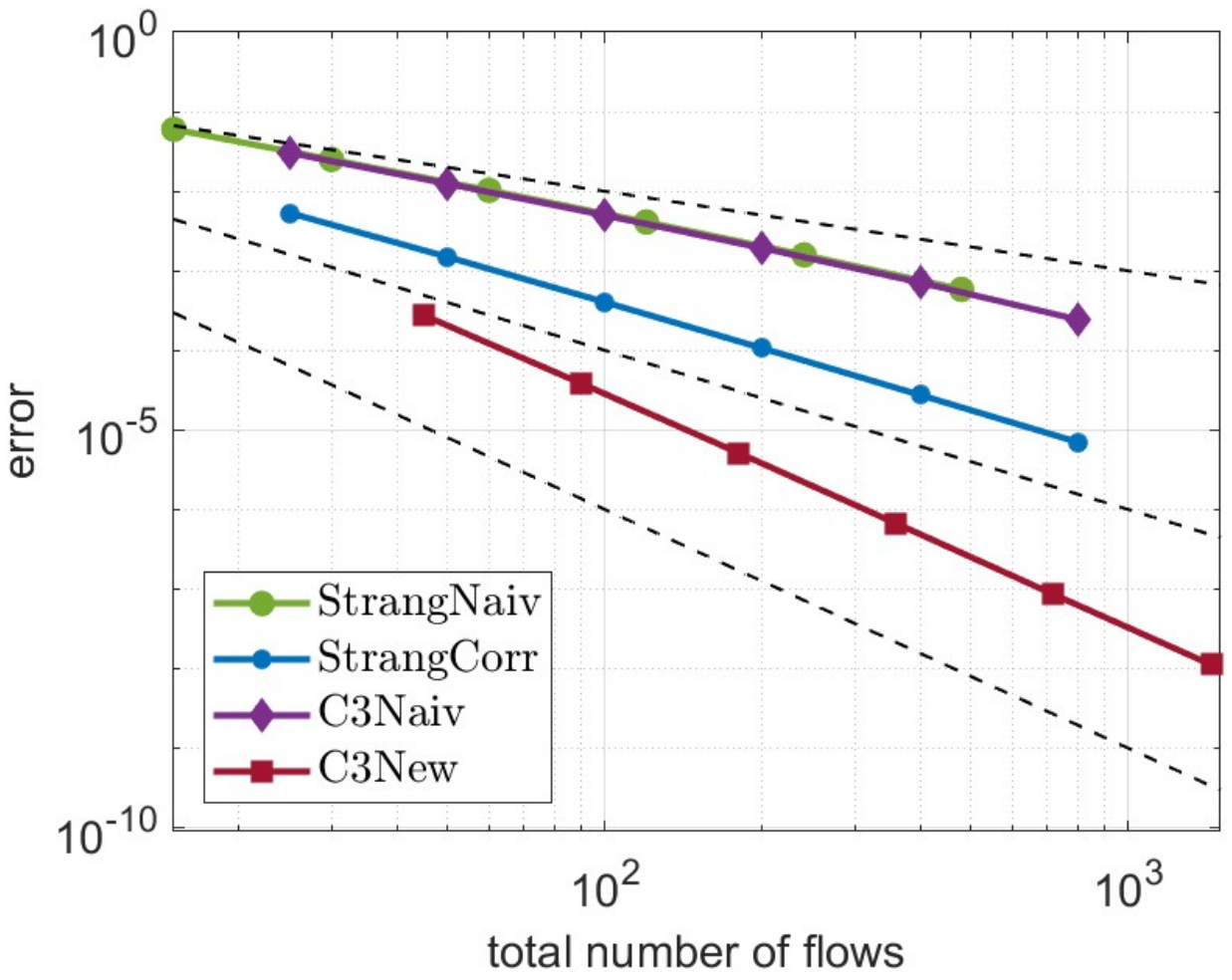}
\end{minipage}
\vspace*{1mm}
\caption{Convergence error of the splitting methods \emph{StrangNaiv}, \emph{StrangCorr}, \emph{C3Naiv} and \emph{C3New} in dependence of the time step (left) and of the total number of flows (right), applied to problem \eqref{pdeindep2d}. Reference slopes of order one, two and three are given in dotted lines.
}
\label{fig:indf_2d}
\end{figure}

In Figure \ref{fig:indf_2d}, we compare the convergence of the new method \emph{C3New} with the splitting methods \emph{Strang\-Naiv}, \emph{StrangCorr} and \emph{C3Naiv}. We observe that \emph{C3New} converge with order three while we note a convergence order between one and two for \emph{C3Naiv}. Additionally, for example for a time step $\tau=10^{-3}$, we note in the left picture that the error is $10^4$ times more accurate for the corrected method \emph{C3New}, compared to \emph{C3Naiv}. On the right picture, we recover that \emph{C3New} requires less evaluations of the flows for a given precision and is therefore more accurate than the other considered splitting methods \emph{StrangNaiv}, \emph{StrangCorr} and \emph{C3Naiv}. Additionally, the error in dependence of the total number of flows coincide for the naive methods.

We process this experiment with the same time parameters $\tau, \tau_{ref}$ and $T$ as in the first example. We choose a uniform mesh of size $\Delta x = 10^{-2}$ for the two-dimensional square domain $\Om = (0,1)^2$. 

\paragraph{The Fisher-KPP equation (Figures \ref{fig:fisher} and \ref{fig:fisher:err})} For a third example, we apply the methods \emph{C3Naiv} and its corrected version \emph{C3New} to the Fisher-Kolmogorov-Petrovski-Piskunov equation~\cite{Ach22, Sim06}
\begin{equation} \label{pde:fisher}
\begin{split}
\partial_tu(x,y,t) &= \Delta u(x,y,t)+ F(u(x,y,t)) \quad \text{in}\; \Om\times(0,T], \\ 
u(x,y,t)&=1/2 \quad \text{on}\; \partial\Om\times(0,T],\\
u_0(x,y)&=\sin(2\pi x)\sin(2\pi y)+1/2 \quad \text{in}\; \Om,
\end{split}
\end{equation}
with nonlinearity $F(u) = Mu(1-u)$ for $M>0$. Problem~\eqref{pde:fisher} models the wound heeling process when only biological feature are considered. The solution $u$ represents the cell density.
\begin{figure}[!tbp] 
\begin{minipage}[b]{0.49\textwidth}
\includegraphics[clip,trim=100 60 120 60,scale=0.36]{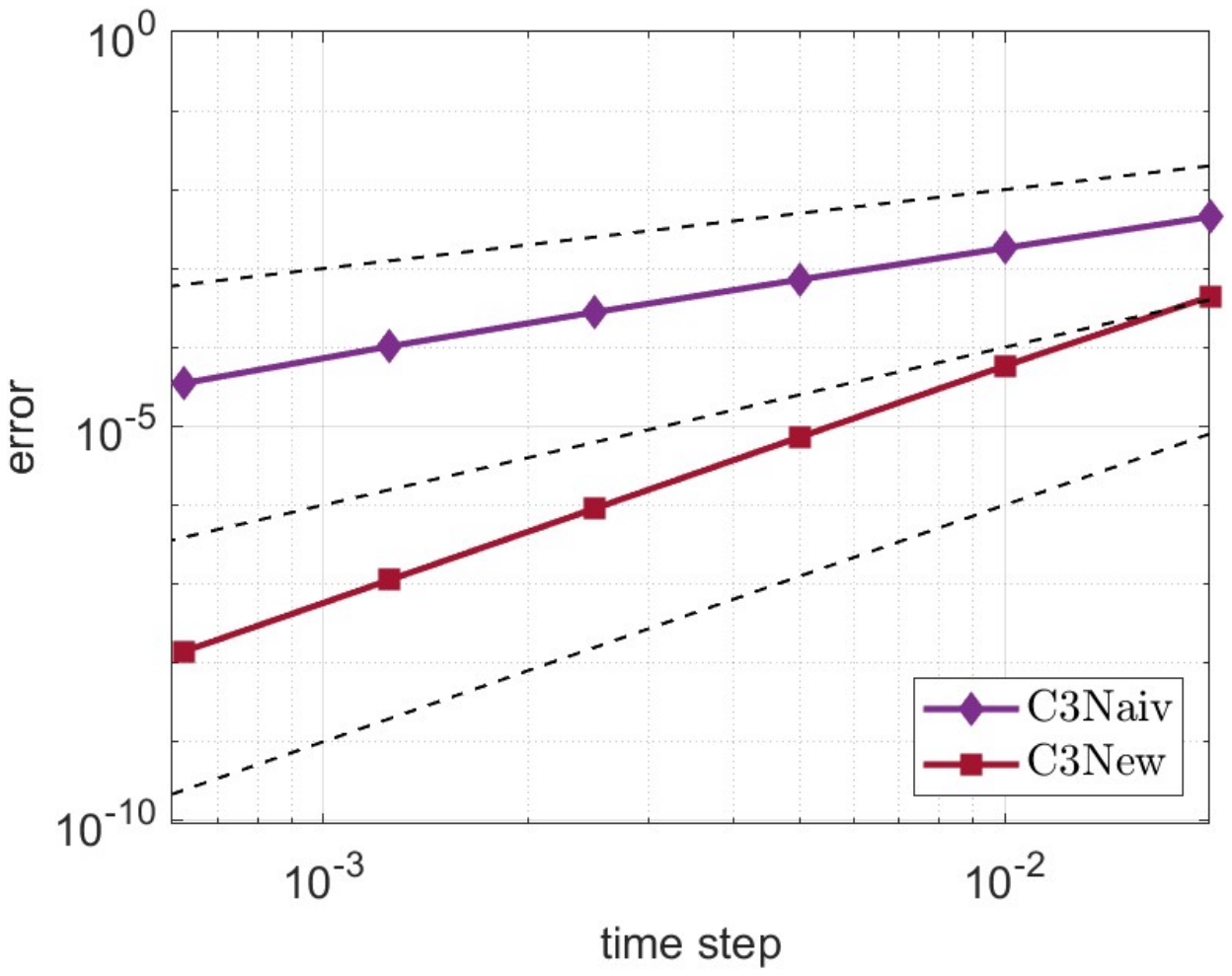}
\end{minipage}
\hspace{-10pt}
\begin{minipage}[b]{0.49\textwidth}
\includegraphics[clip,trim=100 60 120 60,scale=0.36]{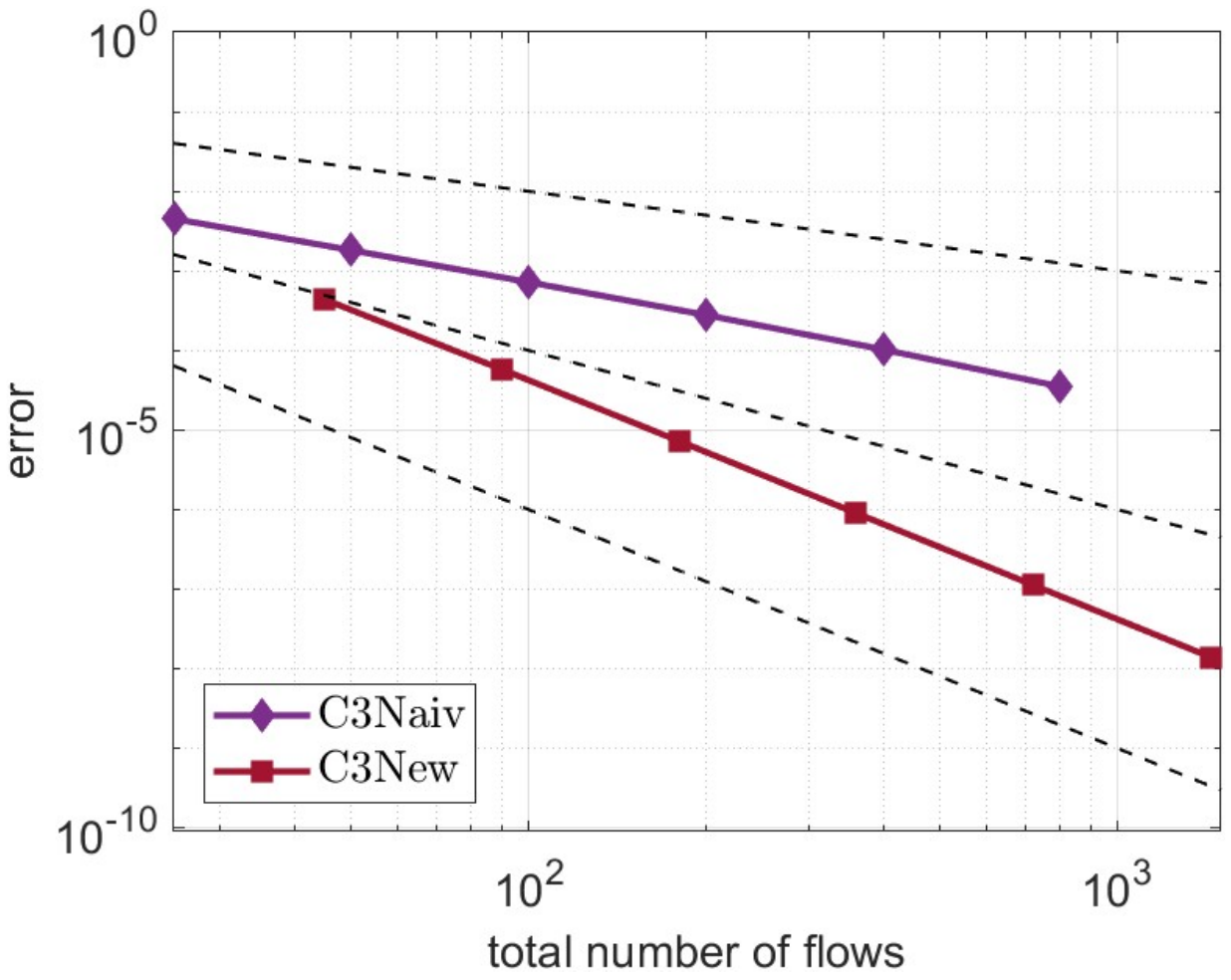}
\end{minipage}
\vspace*{1mm}
\caption{Convergence error of the splitting methods \emph{C3Naiv} and \emph{C3New} in depedence of the time step (left) and of the total number of flows (right), applied to problem \eqref{pde:fisher} for a cell proliferation rate $M=1$. Reference slopes of order one, two and three are given in dotted lines.
}
\label{fig:fisher}
\end{figure}

We compare in Figure \ref{fig:fisher} the convergence of the splitting methods \emph{C3Naiv} and \emph{C3New}. For $M=1$, we observe third order convergence of the modified splitting method \emph{C3New}, while the naive method \emph{C3Naiv} converges with order one and thus, suffers from an order reduction. Furthermore, \emph{C3New} has a better error constant compared to \emph{C3Naiv}. In the right picture, we note additionally that the new method \emph{C3New} requires less evaluations of the flows for a given precision and is therefore more accurate than the naive method \emph{C3Naiv}.

\begin{figure}[!tbp] 
\begin{minipage}[b]{0.49\textwidth}
\includegraphics[clip,trim=100 65 120 65,scale=0.36]{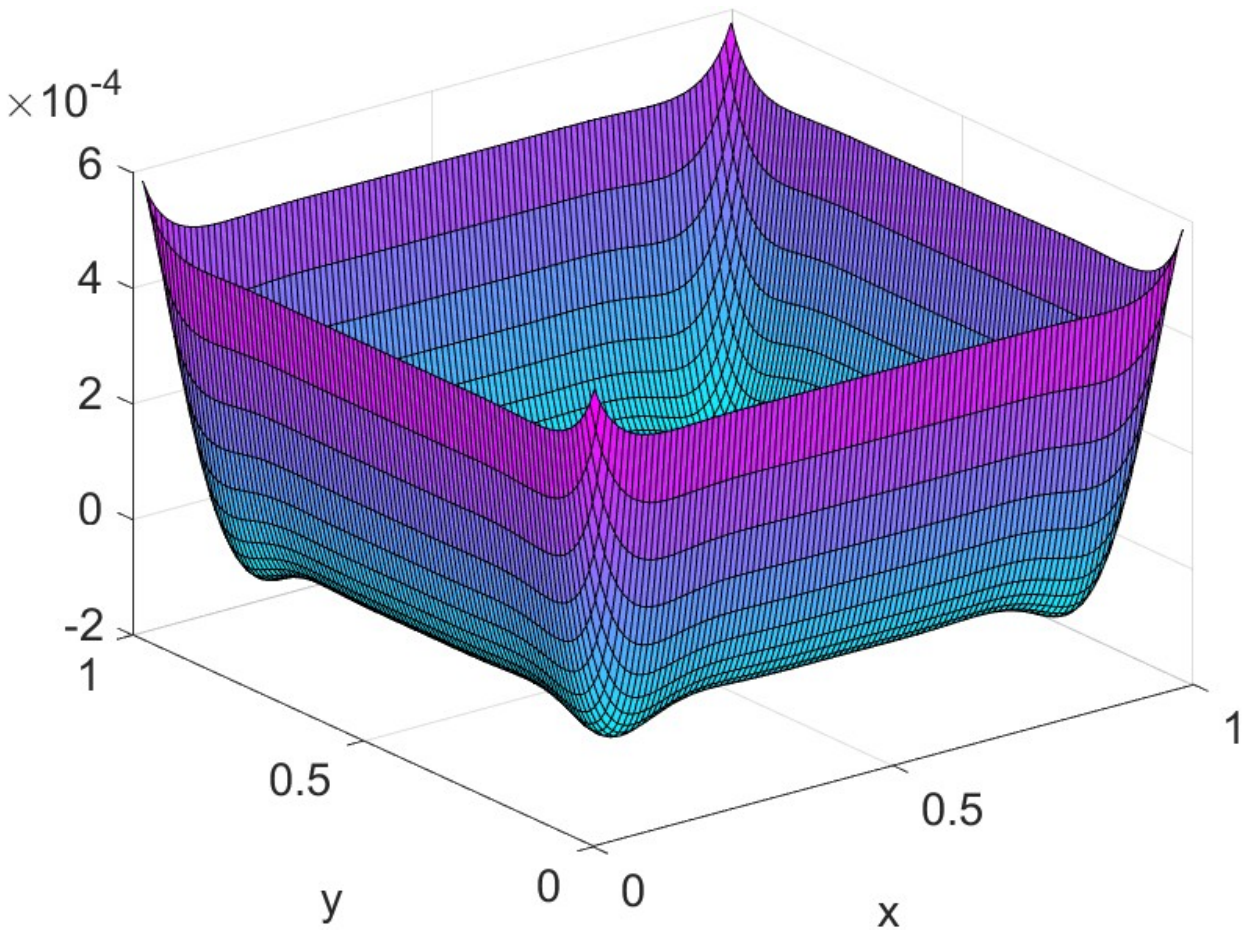}
\end{minipage}
\hspace{-10pt}
\begin{minipage}[b]{0.49\textwidth}
\includegraphics[clip,trim=100 65 120 65,scale=0.36]{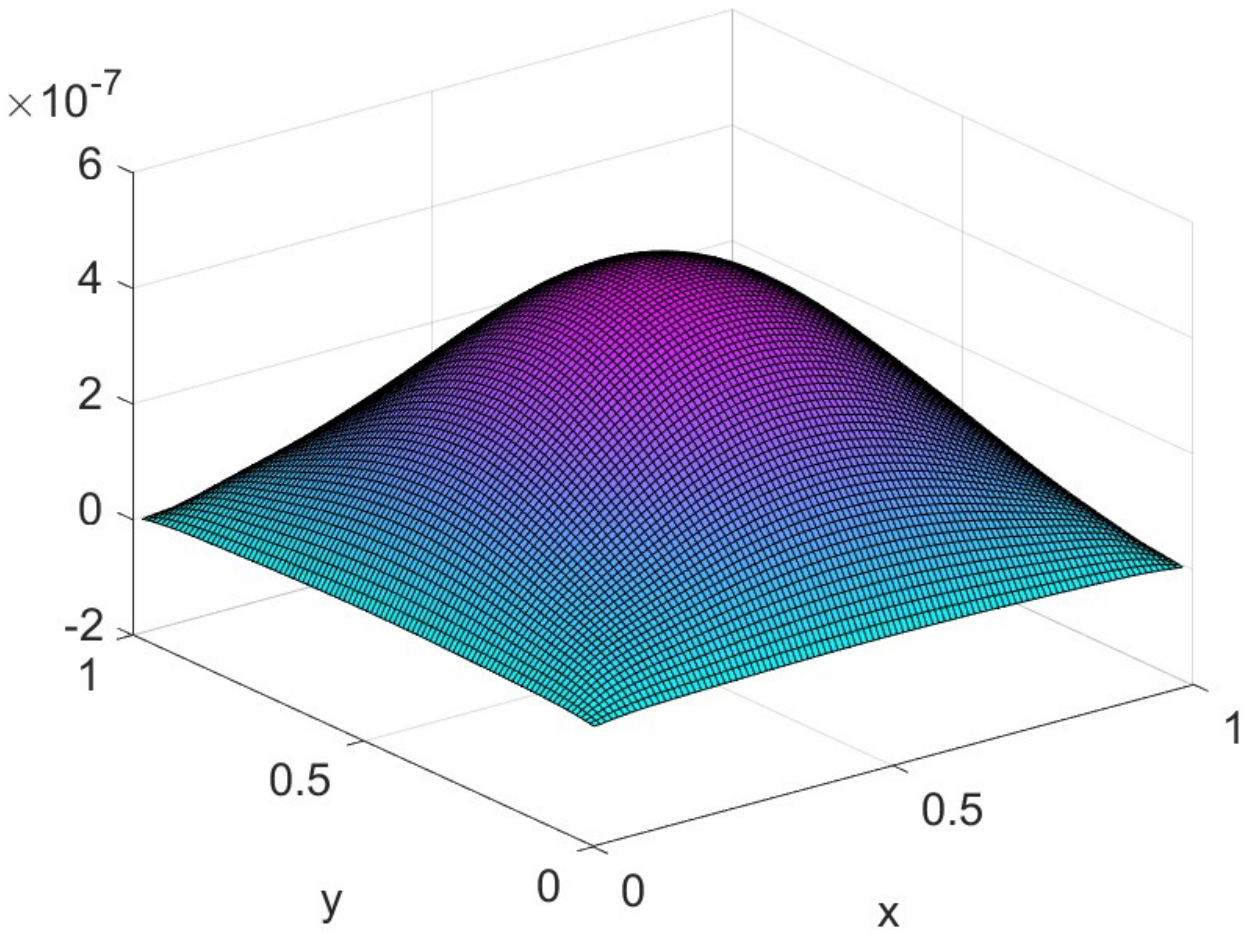}
\end{minipage}
\vspace*{1mm}
\caption{The pointwise error of the methods \emph{C3Naiv} (left) and \emph{C3New} (right) at final time $T=0.1$ applied to problem~\eqref{pde:fisher} in the domain $\Om = (0,1)^2$ for time step $\tau = 10^{-2}$. The error of the naive splitting is concentrated on the boundary $\partial\Om$, where the error of the new method vanishes.
}
\label{fig:fisher:err}
\end{figure}
In Figure~\ref{fig:fisher:err}, we plot the pointwise error in the domain $\Om$ of the naive method \emph{C3Naiv} and its modification \emph{C3New} for fixed time step $\tau=10^{-2}$. On the left picture, we recover that the error of \emph{C3Naiv} is concentrated on the boundary of the domain $\Om$, which was already observed in \cite{Han09b} for the Strang splitting method \emph{StrangNaiv}. In contrast, on the right picture, we note that the error of \emph{C3New} is reduced by three orders of magnitude (from $10^{-4}$ to $10^{-7}$) and is located in the interior of $\Om$ and vanishes on the boundary. 
Therefore, the observation we made in Figure~\ref{fig:indf_1d_2} for a solution-independent source term seems to persist in the context of a nonlinear $f$.

For the implementation of this experiment, we use the same final time $T$ and time steps $\tau, \tau_{ref}$ as in the first experiment. Furthermore, we choose the same discretization for the domain $\Om=(0,1)^2$ as in the second experiment.

\begin{remark} \label{rem_num} If we consider solution-dependent source terms $f=f(u)$, numerical differentiation is required to get boundary conditions~\eqref{r} and~\eqref{q}. In particular, again for $N+1=1/\Delta x$, we consider the uniform mesh $\wt{\Om}=(x_i,y_k)_{i,k=1}^n$ of $\Om$.  
Then, denoting $v_n^{(j)} =  \phi^f_{\tau_j}(\omega^{(j)}_n)-\omega^{(j)}_n$, to evaluate values on the boundary, we use the finite difference scheme
\begin{equation} \label{lapl_r}
\partial_{xx} v^{(j)}_n (0,y_k) \approx \frac{2v^{(j)}_n(0,y_k)-5v^{(j)}_n(x_1,y_k)+3v^{(j)}_n(x_2,y_k)-v^{(j)}_n(x_3,y_k)}{\Delta x^2}, \quad k=1, \ldots N,
\end{equation}
and analogousy for $\partial_{xx}v^{(j)}_n(1,y_k)$ as well as for $\partial_{yy}v^{(j)}_n(x_i,0)$ and $\partial_{yy}v^{(j)}_n(x_i,1), i=1, \ldots, N$. Additionally, we approximate $\omega^{(j)}_n \approx b$ on $\partial\Om$. Similarly to Remark~\ref{rem_numf}, we use finite differences to implement the Laplacian $\Delta = \partial_{xx}+\partial_{yy}$. Then, for $i, k = 1, \ldots, N$, we obtain the corrector functions $q^{(j)}_n$ by means of the identities
\begin{align*}
\Delta r_n^{(j)}(x_i,y_k)+\delta_{i1}\partial_{xx} v^{(j)}_n (0,y_k)+\delta_{iN}\partial_{xx} v^{(j)}_n (1,y_k)+\delta_{k1}\partial_{yy} v^{(j)}_n (x_i,0) +\delta_{kN}\partial_{yy} v^{(j)}_n (x_i,1)&=0
\end{align*}
and
\begin{align*}
\Delta q^{(j)}_n(x_i,y_k)+\delta_{i1} v^{(j)}_n (0,y_k)+\delta_{iN}v^{(j)}_n (1,y_k)+\delta_{k1}v^{(j)}_n (x_i,0) +\delta_{kN}v^{(j)}_n (x_i,1) &=r_n^{(j)}(x_i,y_k).
\end{align*}
\end{remark}

\paragraph{Solution-dependent boundary conditions (Figure~\ref{fig:autbc})} To emphasize that the new method \emph{C3New} performs well also for solution-dependent boundary conditions $b(u)$, we consider the following nonlinear problem in $\Om=(0,1)$,
\begin{equation} \label{pdetimeaut}
\begin{split}
\partial_tu(x,t) &= \partial_{xx}u(x,t)+ \cos(u(x,t)) \quad \text{in}\; \Om\times(0,T],  \\
u(0,t)&=\int^1_0 \rho(x)u(x,t)\;\text{d}x, \; u(1,t)=0 \quad  \text{in}\; (0,T],\\
u_0(x)&=\sin(2\pi x) \quad \text{in}\; \Om,
\end{split}
\end{equation}
\begin{figure}[!tbp] 
\begin{minipage}[b]{0.49\textwidth}
\includegraphics[clip,trim=100 60 120 60,scale=0.36]{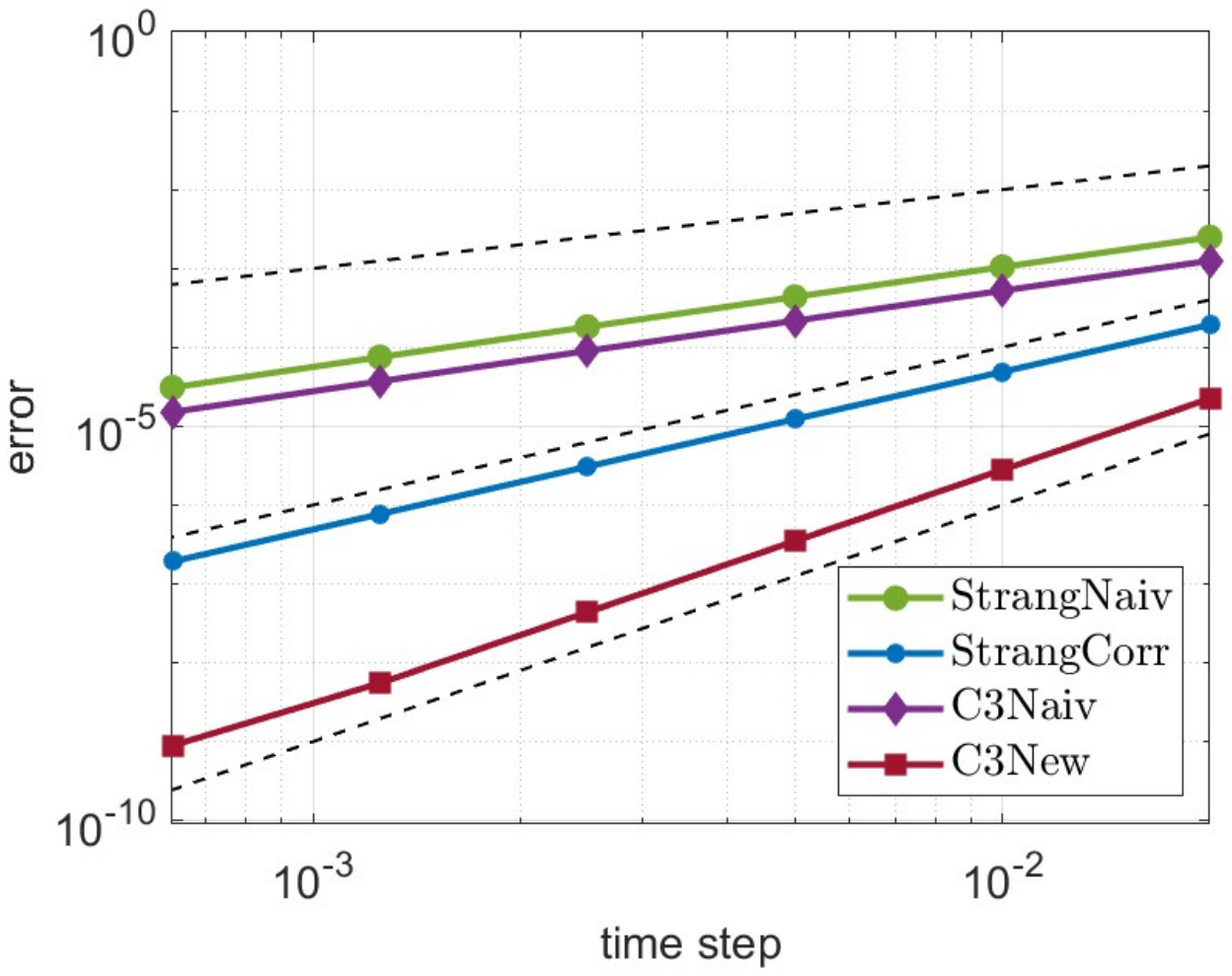}
\end{minipage}
\hspace{-10pt}
\begin{minipage}[b]{0.49\textwidth}
\includegraphics[clip,trim=100 60 120 60,scale=0.36]{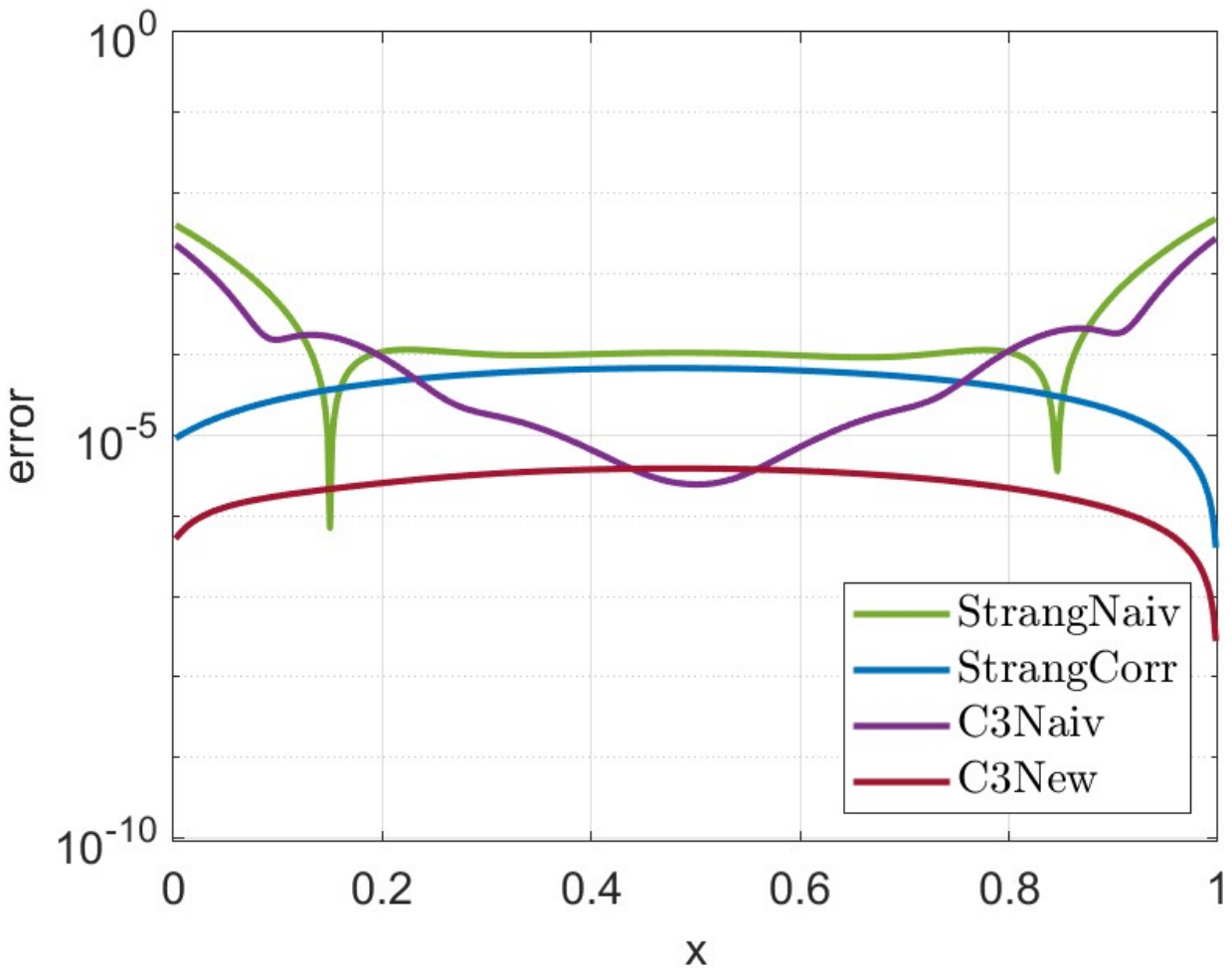}
\end{minipage}
\vspace*{1mm}
\caption{Convergence error of the corrected method \emph{C3New} applied to problem \eqref{pdetimeaut}. Comparison to the naive third order method \emph{C3Naiv}, the Strang splitting \emph{StrangNaiv} and its corrected version \emph{StangCorr}. Left: Error in terms of the time step $\tau$. Reference slopes of order one, two and three are given in dotted lines. Right: The pointwise error for time step $\tau = 10^{-2}$.
}
\label{fig:autbc}
\end{figure}
with $\rho(x)=-(x-1/2)^2+1/4$. In Figure~\ref{fig:autbc}, we compare the convergence of the corrected schemes \emph{StrangCorr}, \emph{C3New} with their naive versions \emph{StrangNaiv}, \emph{C3Naiv} respectively. We observe that the new method \emph{C3New} converges with full order three when integrating problem~\eqref{pdetimeaut} and thus does not suffer from an order reduction, in contrast to the naive schemes \emph{StrangNaiv} and \emph{C3Naiv}. Moreover, it has a better error constant compared to the other splitting methods we processed. Furthermore, as already observed for problem~\eqref{pdeindep}, we note that the error for the methods \emph{StrangNaiv} and \emph{C3Naiv} is maximal at the boundary of the domain, for fixed time step $\tau =10^{-2}$, see Figure~\ref{fig:autbc} (right). In contrast, the corrected methods \emph{StrangCorr} and \emph{C3New} are much more accurate on the boundary $\partial\Om$. We conclude that, even if the convergence analysis of Section~\ref{conv_anal} does not apply in the context of problem~\eqref{pdetimeaut}, the third order convergence of \emph{C3New} seems to persist for solution-dependent boundary conditions.

For the performance of this experiment, we use the same time parameters $\tau$ and $T$ as in the previous examples.  The flow of the source term $\cos(u)$ is obtained by the fourth order Runge Kutta method RK4 with time step $\tau_f = 0.02\cdot2^{-14} \approx 10^{-6}$. The domain $\Om = (0,1)$ is discretized in a uniform mesh with mesh size $\Delta x = 2\cdot10^{-3}$, we refer to Remark~\ref{rem_numf} for the notation. The Laplacian $\partial_{xx}$ was approximated by finite differences, while we use the trapezoidal rule to implement the solution of~\eqref{pdetimeaut} in $x=0$. For the corrector functions $q_n^{(j)}$, we set boundary conditions
\begin{equation*}
q_n^{(j)}(0) = v_n^{(j)}(0) - \Delta x \sum_{i=1}^N \rho(x_i)v_n^{(j)}(x_i) \quad \text{and} \quad q_n^{(j)}(1) = v_n^{(j)}(1),
\end{equation*}
where we denote $v_n^{(j)} =  \phi^f_{\tau_j}(\omega^{(j)}_n)-\omega^{(j)}_n$. Furthermore, we consider $r_n^{(j)}=\partial_{xx} v^{(j)}_n$ on $\partial\Om$, where we use the approximation~\eqref{lapl_r} without the dependence on $y_k$.
The reference solution corresponds to \emph{StrangNaiv}, with a very small time step $\tau_{ref} = 0.02\cdot2^{-15}$.

These numerical results suggest that the new method \emph{C3New} converges with order three also in the context of a solution-dependent source term $f$, and seems to avoid order reduction in general when applied to the class of semilinear parabolic problems~\eqref{pde}, even in the context of solution-dependent boundary conditions $b$. We now show numerically that the analysis from Section~\ref{conv_anal} seems to persist for order four by illustrating the convergence of the introduced method \emph{C4New}.

\paragraph{The fourth order splitting \emph{C4New} (Figure~\ref{fig:c4})} For the numerical performance of \emph{C4New}, we consider firstly a parabolic equation in the two-dimensional domain $\Om = (0,1)^2$, with linear source term $f(x,y)$ and homogeneous Dirichlet boundary conditions (Figure~\ref{fig:c4}, left),
\begin{figure}[!tbp]
\begin{minipage}[b]{0.49\textwidth}
\includegraphics[clip,trim=120 60 110 60,scale=0.36]{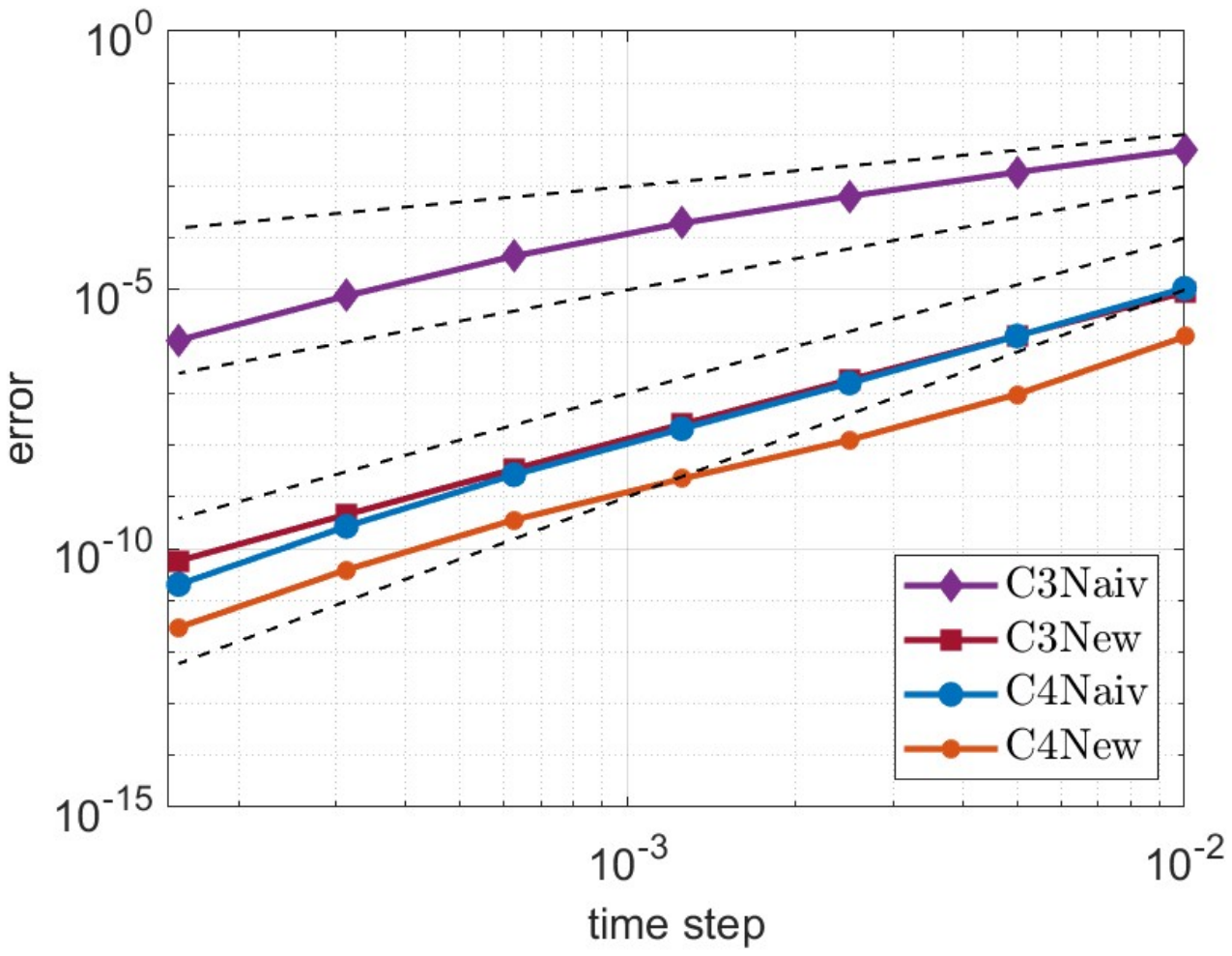}
\centering \small
linear source term $f$
\end{minipage}
\hspace{-10pt}
\begin{minipage}[b]{0.49\textwidth}
\includegraphics[clip,trim=120 60 110 60,scale=0.36]{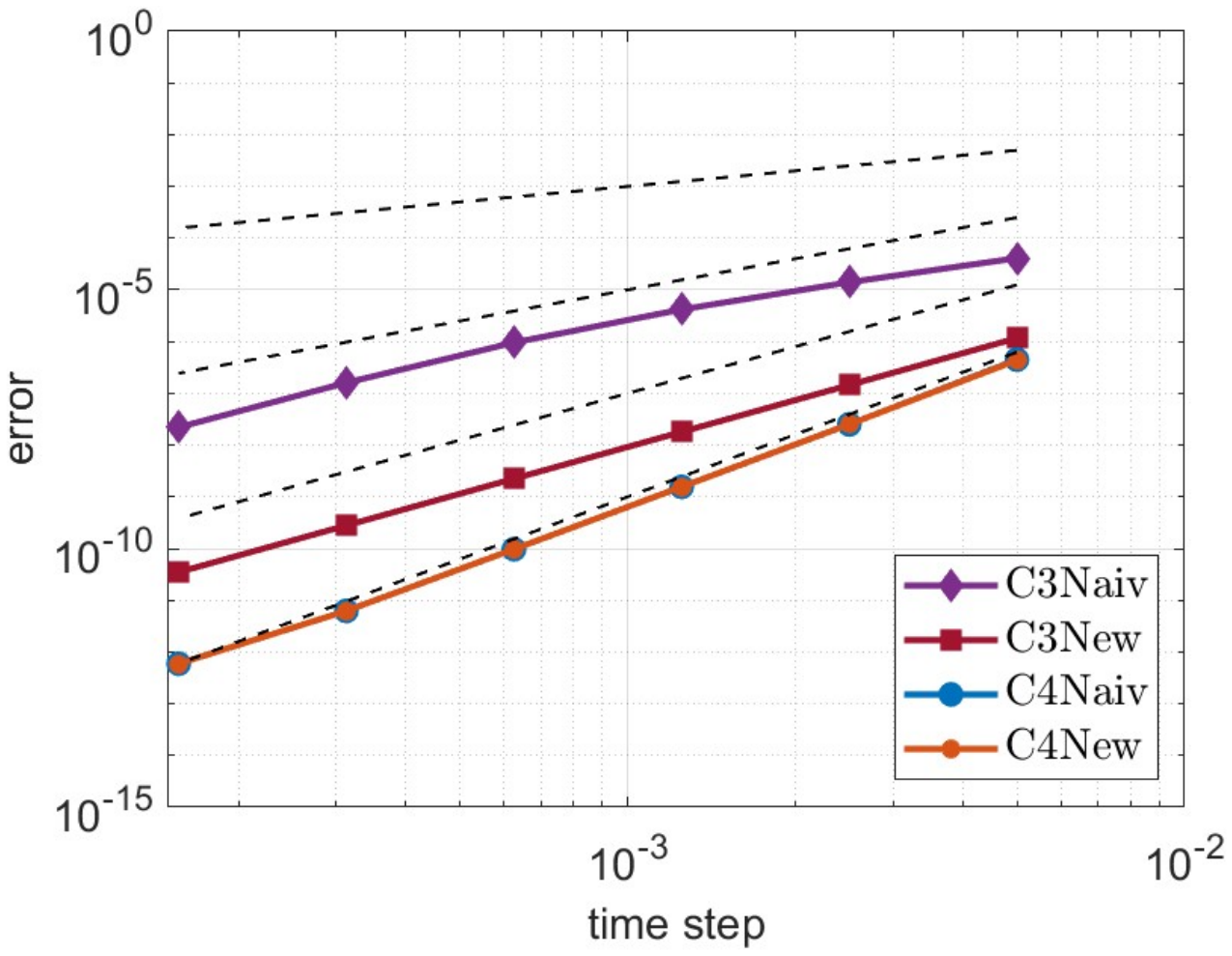}
\centering \small
nonlinear source term $f$
\end{minipage}
\vspace*{1mm}
\caption{Convergence error of the corrected method \emph{C4New} applied to problem \eqref{pde:c4} with solution-independent source function $f(x,y)=\cos(2\pi x)y^7+1$ in 2D (left) and to the Fisher equation~\eqref{pde:fisher} with nonlinearity $f(u(x))=u(x)(1-u(x))$ in 1D (right). Comparison to the naive method \emph{C4Naiv}, the splitting scheme \emph{C3Naiv} and its  corrected version \emph{C3New}. Reference slopes of order 1, 2, 3 and 4 are given in dotted lines.
}
\label{fig:c4}
\end{figure}
\begin{equation}
\begin{split}
\partial_tu(x,t) &= \Delta u(x,t)+ \cos(2\pi x)y^7+1 \quad \text{in}\; \Om\times(0,T], \\ 
u(x,t)&=0 \quad \text{on}\; \partial\Om\times(0,T] ,\\
u(x,0)&=\sin(2\pi x)\sin(2\pi y) \quad \text{in}\; \Om.
\label{pde:c4}
\end{split}
\end{equation}

We compare the convergence of the splitting methods \emph{C3Naiv} and \emph{C4Naiv} with their corrected versions \emph{C3New}, \emph{C4New} respectively. We observe that \emph{C4New} converges with order four, while its naive version \emph{C4Naiv} suffers from an order reduction and is not more accurate than \emph{C3New}. Around time step $\tau = 10^{-3}$, we note a switch of the error constant of \emph{C4New}, which is possibly caused by the transition from the ODE regime to the PDE regime. Compared to \emph{C3New}, the new method is 10 times more accurate. 

In a second experiment (Figure~\ref{fig:c4}, right), we illustrate the convergence of method \emph{C4New} for the Fisher-KPP equation~\eqref{pde:fisher} in $\Om =(0,1)$ and compare it to the splitting schemes \emph{C3Naiv}, \emph{C3New} and \emph{C4Naiv}. We note that the new methods \emph{C3New} and \emph{C4New} converge with full order three, four respectively. In this case, also the naive method \emph{C4Naiv} avoids order reduction and we obtain the same error constant as for \emph{C4New}. As already observed in dimension two (Figure~\ref{fig:fisher}), \emph{C3Naiv} suffers from an order reduction and converge with reduced order one.

Analogous to the previous experiments, we set the final time $T=0.1$ and choose time steps $\tau=~0.02\cdot~2^{-k}$ for $k \in \{1, \ldots, 7\}$. For the reference solution, we use the corrected method \emph{C3New} with a small time step $\tau_{ref} = 10^{-5}$. The one-dimensional domain $\Om = (0,1)$ as well as the two-dimensional square $\Om = (0,1)^2$ are discretized in a uniform mesh with mesh size $\Delta x = 2\cdot10^{-2}$. 

These numerical experiments suggest that the correction techniques, which where introduced in~\cite{Ber20, Ein15, Ein16} for order two, and in the present aricle for order three, can be generalized to construct higher order splitting methods which integrate in time the class of semilinear parabolic problems~\eqref{pde} without order reduction.

\section*{Conclusion}
We introduce a new splitting scheme what overcomes the order barrier two when integrating in time semilinear parabolic problems in the context of non-periodic boundary conditions. While we show third order convergence in the setting where the source term does not depend on the solution of the problem, numerical experiments suggest that the third order convergence remains true for nonlinear functions. Moreover, we observe numerically that the correction techniques introduced for order three, can be generalized to get a scheme of order four.\\
Such techniques might be generalized to compute splitting schemes of arbitrarily high order. Another direction of interest and topic of future work is to study high order splitting methods for hyperbolic problems, where the parabolic regularity properties are not available.

\section*{Acknowledgements}
The author would like to thank Gilles Vilmart for helpful discussions. This work was partially supported by the Swiss National Science
Foundation, projects No 200020\_214819, No. 200020\_184614, and
No. 200020\_192129.

\bibliographystyle{plain}

\bibliography{references_arxiv}

@article {Ach22,
    AUTHOR = {Achouri, T. and Ayadi, M. and Habbal, A. and Yahyaoui, B.},
     TITLE = {Numerical analysis for the two-dimensional {F}isher-{K}olmogorov-{P}etrovski-{P}iskunov equation with mixed boundary condition},
   JOURNAL = {J. Appl. Math. and Comput.},
  FJOURNAL = {Journal of Applied Mathematics and Computing},
    VOLUME = {68},
      YEAR = {2022},
    NUMBER = {6},
     PAGES = {3589--3614},
      ISSN = {1064-8275},
   MRCLASS = {65J08 (65L04 65M12)},
  MRNUMBER = {4117850},
       DOI = {10.1007/s12190-021-01679-7},
       URL = {https://doi.org/10.1007/s12190-021-01679-7},
}

@article {Alo17,
    AUTHOR = {Alonso-Mallo, I. and Cano, B. and Reguera, N.},
     TITLE = {Avoiding order reduction when integrating linear initial boundary value problems with {L}awson methods},
   JOURNAL = {IMA J. Num. Anal.},
  FJOURNAL = {Journal of Numerical Analysis},
    VOLUME = {37},
      YEAR = {2017},
    NUMBER = {},
     PAGES = {2091--2119},
      ISSN = {},
   MRCLASS = {},
  MRNUMBER = {},
       DOI = {10.1093/imanum/drw052},
       URL = {https://doi.org/10.1093/imanum/drw052},
}

@article {Alo19,
    AUTHOR = {Alonso-Mallo, I. and Cano, B. and Reguera, N.},
     TITLE = {Avoiding order reduction when integrating reaction-diffusion boundary value problems with exponential splitting methods},
   JOURNAL = {J. Comput. and Appl. Math.},
  FJOURNAL = {Journal of Computational and Applied Mathematics},
    VOLUME = {357},
      YEAR = {2019},
    NUMBER = {},
     PAGES = {228--250},
      ISSN = {},
   MRCLASS = {},
  MRNUMBER = {},
       DOI = {10.1016/j.cam.2019.02.023},
       URL = {https://doi.org/10.1016/j.cam.2019.02.023},
}

@article {Alo20,
    AUTHOR = {Alonso-Mallo, I. and Cano, B. and Reguera, N.},
     TITLE = {Comparison of efficiency among different techniques
to avoid order reduction with {S}trang splitting},
   JOURNAL = {Num. Methods PDEs, Wiley},
  FJOURNAL = {Numerical Methods for Partial Differential Equations, Wiley online Library},
    VOLUME = {37},
      YEAR = {2021},
    NUMBER = {1},
     PAGES = {854--873},
      ISSN = {},
   MRCLASS = {},
  MRNUMBER = {},
       DOI = {10.1002/num.22556},
       URL = {https://doi.org/10.1002/num.22556},
}

@article {Alo21,
    AUTHOR = {Alonso-Mallo, I. and Portillo, A. M.},
     TITLE = {Integrating semilinear wave problems with time-dependent
boundary values using arbitrarily high-order splitting methods},
   JOURNAL = {MDPI Mathematics},
  FJOURNAL = {MDPI Mathematics},
    VOLUME = {9},
      YEAR = {2021},
    NUMBER = {10},
     PAGES = {854--873},
      ISSN = {},
   MRCLASS = {},
  MRNUMBER = {65M12; 65M20; 65M22},
       DOI = {10.3390/math9101113},
       URL = {https://doi.org/10.3390/math9101113},
}

@article {Bad13,
    AUTHOR = {Bader, P. and Blanes, S. and Casas, F.},
     TITLE = {Solving the {S}chrödinger eigenvalue problem by the imaginary time propagation technique using splitting methods with complex coefficients},
   JOURNAL = {J. Chem. Phys.},
  FJOURNAL = {The Journal of Chemical Physics},
    VOLUME = {139},
      YEAR = {2013},
    NUMBER = {124117},
     PAGES = {1--11},
      ISSN = {},
   MRCLASS = {},
  MRNUMBER = {},
       DOI = {10.1063/1.4821126},
       URL = {https://doi.org/10.1063/1.4821126},
}

@article {Ber20,
    AUTHOR = {Bertoli, G. and Vilmart, G.},
     TITLE = {Strang splitting method for semilinear parabolic problems with
              inhomogeneous boundary conditions: a correction based on the
              flow of the nonlinearity},
   JOURNAL = {SIAM J. Sci. Comput.},
  FJOURNAL = {SIAM Journal on Scientific Computing},
    VOLUME = {42},
      YEAR = {2020},
    NUMBER = {3},
     PAGES = {A1913--A1934},
      ISSN = {1064-8275},
   MRCLASS = {65J08 (65L04 65M12)},
  MRNUMBER = {4117850},
       DOI = {10.1137/19M1257081},
       URL = {https://doi.org/10.1137/19M1257081},
}

@article {Bla05,    
AUTHOR = {Blanes, S. and Casas, F.},      
TITLE = {On the necessity of negative coefficients for operator splitting schemes of order higher than two},    
JOURNAL = {Appl. Num. Math.},   FJOURNAL = {Applied Numerical Mathematics},     
VOLUME = {54},       
YEAR = {2005},     
NUMBER = {1},      
PAGES = {23--37},       
ISSN = {0377-0427},    
MRCLASS = {37M15 (34A45 65L06 65P10)},  
MRNUMBER = {1906732},        
DOI = {10.1016/j.apnum.2004.10.005},        
URL = {https://doi.org/10.1016/j.apnum.2004.10.005}, }

@article {Bla13,    
AUTHOR = {Blanes, S. and Casas, F and Chartier, P. and Murua, A.},      
TITLE = {Optimized high order splitting methods for some classes of parabolic equations},    
JOURNAL = {Math. of Comput.},   FJOURNAL = {Mathematics of Computation},     
VOLUME = {82},       
YEAR = {2013},     
NUMBER = {283},      
PAGES = {1559--1576},       
ISSN = {0377-0427},    
MRCLASS = {37M15 (34A45 65L06 65P10)},  
MRNUMBER = {1906732},        
DOI = {10.48550/arXiv.1102.1622},        
URL = {https://doi.org/10.48550/arXiv.1102.1622}, }

@article {Can18,
    AUTHOR = {Cano, B. and Moreta, M. J.},
     TITLE = {Exponential quadrature rules without order reduction for integrating linear initial boundary value problems},
   JOURNAL = {SIAM J. Numer. Anal.},
  FJOURNAL = {SIAM Journal on Numerical Analysis},
    VOLUME = {56},
      YEAR = {2018},
    NUMBER = {3},
     PAGES = {1187--1209},
      ISSN = {},
   MRCLASS = {65M12 (65M20)},
  MRNUMBER = {},
MRREVIEWER = {},
       DOI = {10.1137/17M1124279},
       URL = {https://doi.org/10.1137/17M1124279},
}

@article {Cas09,
    AUTHOR = {Castella, F. and Chartier, P. and Descombes, S. and Vilmart, G.},
     TITLE = {Splitting methods with complex times for parabolic equations},
   JOURNAL = {Bit Num. Math.},
  FJOURNAL = {SIAM Journal on Numerical Analysis},
    VOLUME = {49},
      YEAR = {2009},
    NUMBER = {},
     PAGES = {487--508},
      ISSN = {0036-1429},
   MRCLASS = {47D05 (41A20)},
  MRNUMBER = {537280},
MRREVIEWER = {G. W. Hedstrom},
       DOI = {10.1007/s10543-009-0235-y},
       URL = {https://doi.org/10.1007/s10543-009-0235-y},
}

@article {Des01,     
AUTHOR = {Descombes, S.},      
TITLE = {Convergence of a splitting method of high order for reaction-diffusion systems},    
JOURNAL = {Math. of Comput.},   
FJOURNAL = {Mathematics of Computation},     
VOLUME = {70},       
YEAR = {2001},    
 NUMBER = {236},      
PAGES = {1481--1501},      
 ISSN = {0254-9409},    
MRCLASS = {65M06},  
 MRNUMBER = {2359962}, 
MRREVIEWER = {Sonia Busquier}, 
DOI = {10.1090/S0025-5718-00-01277-1},
URL = {https://doi.org/10.1090/S0025-5718-00-01277-1},}

@article {Ein18,
    AUTHOR = {Einkemmer, L. and Moccaldi, M. and Ostermann,
              A.},
     TITLE = {Efficient boundary corrected {S}trang splitting},
   JOURNAL = {Appl. Math. Comput.},
  FJOURNAL = {Applied Mathematics and Computation},
    VOLUME = {332},
      YEAR = {2018},
     PAGES = {76--89},
      ISSN = {0096-3003},
   MRCLASS = {65J08 (35K20 35K58)},
  MRNUMBER = {3788673},
       DOI = {10.1016/j.amc.2018.03.006},
       URL = {https://doi.org/10.1016/j.amc.2018.03.006},
}

@article {Ein15,
    AUTHOR = {Einkemmer, L. and Ostermann, A.},
     TITLE = {Overcoming order reduction in diffusion-reaction splitting. Part 1: Dirichlet boundary conditions},
   JOURNAL = {SIAM J. on Sc. Comput.},
  FJOURNAL = {SIAM Journal on Scientific Computing},
    VOLUME = {37},
 NUMBER = {3},
      YEAR = {2015},
     PAGES = {A1577--A1592},
      ISSN = {0096-3003},
   MRCLASS = {65J08 (35K20 35K58)},
  MRNUMBER = {3788673},
       DOI = {10.1137/140994204},
       URL = {https://doi.org/10.1137/140994204},
}

@article {Ein16,
    AUTHOR = {Einkemmer, L. and Ostermann, A.},
     TITLE = {Overcoming order reduction in diffusion-reaction splitting. Part 2: Oblique boundary conditions},
   JOURNAL = {SIAM J. on Sc. Comput.},
  FJOURNAL = {SIAM Journal on Scientific Computing},
    VOLUME = {38},
 NUMBER = {6},
      YEAR = {2016},
     PAGES = {A3741--A3757},
      ISSN = {0096-3003},
   MRCLASS = {65J08 (35K20 35K58)},
  MRNUMBER = {3788673},
       DOI = {10.1137/16M1056250},
       URL = {https://doi.org/10.1137/16M1056250},
}

@book {Eva97,
    AUTHOR = {Evans, L. C.},
     TITLE = {Partial differential equations},
    SERIES = {Graduate Studies in Mathematics},
    VOLUME = {19},
 PUBLISHER = {American Mathematical Society},
      YEAR = {1997},
     PAGES = {xxii+586},
      ISBN = {0-387-98463-1},
   MRCLASS = {47D06 (34G10 35K90 47N20)},
  MRNUMBER = {1721989},
MRREVIEWER = {C. J. K. Batty},
}

@article {Fao15,
    AUTHOR = {Faou, E. and Ostermann, A. and Schratz, K.},
     TITLE = {Analysis of exponential splitting methods for inhomogeneous parabolic
equations},
   JOURNAL = {IMA J. of Num. Anal.},
  FJOURNAL = {IMA Journal of Numerical Analysis},
    VOLUME = {35},
      YEAR = {2015},
     PAGES = {161--178},
      ISSN = {},
   MRCLASS = {},
  MRNUMBER = {},
MRREVIEWER = {},
     DOI = {10.1093/imanum/dru002},
       URL = {https://doi.org/10.1093/imanum/dru002},
}

@article {Ger02,
    AUTHOR = {Gerisch, A. and Verwer, J. G.},
     TITLE = {Operator splitting and approximate factorization for taxis-diffusion-reaction models},
   JOURNAL = {Appl. Num. Math.},
  FJOURNAL = {Applied Numerical Mathematics},
    VOLUME = {42},
 NUMBER = {1-3},
      YEAR = {2002},
     PAGES = {159--176},
      ISSN = {0003-9527},
   MRCLASS = {46.38 (47.00)},
  MRNUMBER = {0213864},
MRREVIEWER = {C. Goulaouic},
       DOI = {10.1016/S0168-9274(01)00148-9},
       URL = {https://doi.org/10.1016/S0168-9274(01)00148-9},
}

@article {Gol96,
    AUTHOR = {Goldman, D. and Kaper, T. J.},
     TITLE = {N-th order operator splitting schemes and nonreversible systems},
   JOURNAL = {SIAM J. Numer. Anal.},
  FJOURNAL = {SIAM Journal on Numerical Analysis},
    VOLUME = {33},
 NUMBER = {1},
      YEAR = {1996},
     PAGES = {349--367},
      ISSN = {},
   MRCLASS = {},
  MRNUMBER = {},
MRREVIEWER = {},
       DOI = {},
       URL = {https://www.jstor.org/stable/2158438},
}

@book {Hai1,
    AUTHOR = {Hairer, E. and Lubich, Ch. and Wanner, G.},
     TITLE = {Geometric numerical integration. Structure-preserving algorithms for ordinary differential equations},
    SERIES = {Springer Series in Comp. Math.},
    VOLUME = {31},
      NOTE = {Reprint of the second edition (2006)},
 PUBLISHER = {Springer, Heidelberg},
      YEAR = {2010},
     PAGES = {xviii+644},
      ISBN = {978-3-642-05157-9},
   MRCLASS = {65P10},
}

@article {Han12,
    AUTHOR = {Hansen, E. and Kramer, F. and Ostermann, A.},
     TITLE = {A second-order positivity preserving scheme for semilinear parabolic problems},
   JOURNAL = {Appl. Num. Math.},
  FJOURNAL = {Applied Numerical Mathematics},
    VOLUME = {62},
 NUMBER = {10},
      YEAR = {2012},
    NUMBER = {},
     PAGES = {1428--1435},
      ISSN = {0008-0624},
   MRCLASS = {65M15 (34G10)},
  MRNUMBER = {1740779},
MRREVIEWER = {Stanis\l aw Burys},
       DOI = {10.1016/j.apnum.2012.06.003},
       URL = {https://doi.org/10.1016/j.apnum.2012.06.003},
}

@article {Han09b,
    AUTHOR = {Hansen, E. and Ostermann, A.},
     TITLE = {High order splitting methods for analytic semigroups exist},
   JOURNAL = {BIT Num. Math.},
  FJOURNAL = {BIT Numer Math},
    VOLUME = {49},
      YEAR = {2009},
    NUMBER = {},
     PAGES = {527--542},
      ISSN = {0008-0624},
   MRCLASS = {65M15 (34G10)},
  MRNUMBER = {1740779},
MRREVIEWER = {Stanis\l aw Burys},
       DOI = {10.1007/s10543-009-0236-x},
       URL = {https://doi.org/10.1007/s10543-009-0236-x},
}

@book {Hun03,
    AUTHOR = {Hundsdorfer, W. and Verwer, J.},
     TITLE = {Numerical solution of time-dependent
              advection-diffusion-reaction equations},
    SERIES = {Springer Series in Computational Mathematics},
    VOLUME = {33},
 PUBLISHER = {Springer-Verlag, Berlin},
      YEAR = {2003},
     PAGES = {x+471},
      ISBN = {3-540-03440-4},
   MRCLASS = {65-02 (65L05 65Mxx)},
  MRNUMBER = {2002152},
MRREVIEWER = {Ian Gladwell},
       DOI = {10.1007/978-3-662-09017-6},
       URL = {https://doi.org/10.1007/978-3-662-09017-6},
}

@article {Lor05,
    AUTHOR = {Lorenzi, L. and Lunardi, A. and Metafune, G. and Pallara, D.},
     TITLE = {Analytic semigroups and reaction-diffusion problems},
   JOURNAL = {Internet Seminar},
  FJOURNAL = {Internet Seminar},
    VOLUME = {},
      YEAR = {2005},
    NUMBER = {},
     PAGES = {},
      ISSN = {0025-5718},
   MRCLASS = {65M99},
  MRNUMBER = {528047},
       DOI = {},
       URL = {},
}

@book {Lun95,
     AUTHOR = {Lunardi, A.},
     TITLE = {Analytic semigroups and optimal regularity in parabolic
              problems},
    SERIES = {Modern Birkh\"{a}user Classics},
 PUBLISHER = {Birkh\"{a}user/Springer Basel AG, Basel},
      YEAR = {1995},
     PAGES = {xviii+424},
      ISBN = {978-3-0348-0556-8; 978-3-0348-0557-5},
   MRCLASS = {47D06 (01A75 34G20 35Kxx 46M35 46N20 47N20 58D25)},
  MRNUMBER = {3012216},
}

@article {McL02,
    AUTHOR = {McLachlan, R. I. and Quispel, G. R. W.},
     TITLE = {Splitting methods},
   JOURNAL = {Acta Numerica},
  FJOURNAL = {Acta Numerica},
    VOLUME = {},
      YEAR = {2002},
    NUMBER = {},
     PAGES = {341--434},
      ISSN = {0029-599X},
   MRCLASS = {65L06 (34A25 37C10 37C80 41A58)},
  MRNUMBER = {3510021},
MRREVIEWER = {John W. Mooney},
       DOI = {10.1017/S0962492902000053},
       URL = {https://doi.org/10.1017/S0962492902000053},}

@book {Nec67,
     AUTHOR = {Ne\u{c}as, J.},
     TITLE = {Direct methods in the theory of elliptic equations},
    SERIES = {Springer Monographs in Mathematics},
 PUBLISHER = {Springer Berlin, Heidelberg},
      YEAR = {1967},
     PAGES = {xvi+372},
      ISBN = {},
   MRCLASS = {},
  MRNUMBER = {},
}

@article {Nie12,
    AUTHOR = {Niesen, J. and Wright, W. M.},
     TITLE = {A {K}rylov subspace algorithm for evaluating
              the {$\varphi$}-functions appearing in exponential integrators},
   JOURNAL = {ACM Trans. Math. Software},
  FJOURNAL = {Association for Computing Machinery. Transactions on
              Mathematical Software},
    VOLUME = {38},
      YEAR = {2012},
    NUMBER = {3},
     PAGES = {Art. 22, 19},
      ISSN = {0098-3500},
   MRCLASS = {65L04 (33E20 65D20)},
  MRNUMBER = {2923559},
       DOI = {10.1145/2168773.2168781},
       URL = {https://doi.org/10.1145/2168773.2168781},
}

@article {Sim06,    
AUTHOR = {Simpson, M. J. and Landman, K. A.},      
TITLE = {Characterizing and minimizing the operator split error for {F}isher's equation},    
JOURNAL = {Appl. Math. Letters 19},   
FJOURNAL = {Applied Mathematics Letters 19},     
VOLUME = {},       
YEAR = {2006},     
NUMBER = {},      
PAGES = {604--612},       
ISSN = {1064-8275},    
MRCLASS = {65M20 (65J08 65L04 65M12)},  
MRNUMBER = {3578030},       
 DOI = {},       
 URL = {}, }

@article {Suz90,    
AUTHOR = {Suzuki, M.},      
TITLE = {Fractal decomposition of exponential operators with applications to many-body theories and {M}onte {C}arlo simulations},    
JOURNAL = {Physics Letters A},   
FJOURNAL = {Physics Letters A},     
VOLUME = {146},       
YEAR = {1990},     
NUMBER = {6},      
PAGES = {319--323},       
ISSN = {},    
MRCLASS = {},  
MRNUMBER = {},       
 DOI = {10.1016/0375-9601(90)90962-N},       
 URL = {https://doi.org/10.1016/0375-9601(90)90962-N}, }

@book {Tho06,
    AUTHOR = {Thom\'{e}e, V.},
     TITLE = {Galerkin finite element methods for parabolic problems},
    SERIES = {Springer Series in Comp. Math.},
    VOLUME = {25},
   EDITION = {Second},
 PUBLISHER = {Springer-Verlag, Berlin},
      YEAR = {2006},
     PAGES = {xii+370},
      ISBN = {978-3-540-33121-6; 3-540-33121-2},
   MRCLASS = {65-02 (65M15 65M60)},
  MRNUMBER = {2249024},
}

@article {Yos90,    
AUTHOR = {Yoshida, H.},      
TITLE = {Construction of higher order symplectic integrators},    
JOURNAL = {Physics Letters A},   
FJOURNAL = {Physics Letters A},     
VOLUME = {150},       
YEAR = {1990},     
NUMBER = {5-7},      
PAGES = {262--268},       
ISSN = {},    
MRCLASS = {},  
MRNUMBER = {},       
 DOI = {10.1016/0375-9601(90)90092-3},       
 URL = {https://doi.org/10.1016/0375-9601(90)90092-3}, }

\end{document}